\documentclass[11pt]{amsart}

\usepackage[colorlinks=true, pdfstartview=FitV, linkcolor=blue, 
      citecolor=blue]{hyperref}
\usepackage{amsfonts, amssymb, amsmath, mdframed, graphicx}

\usepackage{bbm, a4wide}

\DeclareMathAlphabet{\mymathbb}{U}{bbold}{m}{n}

\newcommand{\RR}{\mathbb{R}}
\newcommand{\ZZ}{\ts\mathbb{Z}}
\newcommand{\QQ}{\mathbb{Q}}
\newcommand{\CC}{\mathbb{C}}  
\newcommand{\cL}{\mathcal{L}}
\newcommand{\cT}{\mathcal{T}}
\newcommand{\vG}{\varGamma}
\newcommand{\vL}{\varLambda}
\newcommand{\ft}{\mathfrak{t}}

\newcommand{\ts}{\hspace{0.5pt}}
\newcommand{\nts}{\hspace{-0.5pt}}

\newtheorem{theorem}{Theorem}

\newtheorem{lem}[theorem]{Lemma}
\theoremstyle{definition}

\newtheorem{remark}[theorem]{Remark}

\newcommand{\ee}{\ts\mathrm{e}}

\newcommand{\coinc}{\mathrm{coinc}}
\newcommand{\dens}{\mathrm{dens}}
\newcommand{\vol}{\mathrm{vol}}
\newcommand{\ii}{\ts\mathrm{i}}
\newcommand{\bs}{\boldsymbol}

\newcommand{\bmu}{\bs{\mu}}

\newcommand{\tT}{\mathrm{T}}
\newcommand{\tH}{\mathrm{H}}
\newcommand{\tF}{\mathrm{F}}
\newcommand{\tP}{\mathrm{P}}

\newcommand{\fe}{F}

\newcommand{\oplam}{\mbox{\Large $\curlywedge$}}

\newcommand{\exend}{\hfill$\Diamond$}
\newcommand{\defeq}{\mathrel{\mathop:}=}

\newcommand{\myfrac}[2]{\frac{\raisebox{-2pt}{$#1$}}
  {\raisebox{0.5pt}{$#2$}}}

\title[Dynamics and topology of the Hat family of tilings]{Dynamics
  and topology of the \\[2mm]  Hat family of tilings}

\author{Michael Baake}
\address{Fakult\"at f\"ur Mathematik, Universit\"at Bielefeld, \newline
  \indent  Postfach 100131, 33501 Bielefeld, Germany}
\email{$\{$mbaake,gaehler$\}$@math.uni-bielefeld.de}

\author{Franz G\"{a}hler}

\author{Lorenzo Sadun}
\address{Department of Mathematics, Univeristy of Texas, \newline
  \indent 2515 Speedway, PMA 8.100 Austin, TX 78712, USA}
\email{sadun@math.utexas.edu}

\begin{document}

\begin{abstract}
  The recently discovered Hat tiling \cite{Hat} admits a
  $4$-dimensional family of shape deformations, including the
  $1$-parameter family already known to yield alternate monotiles. The
  continuous hulls resulting from these tilings are all topologically
  conjugate dynamical systems, and hence have the same dynamics and
  topology. We construct and analyze a self-similar element of this
  family called the CAP tiling, and we use it to derive properties of
  the entire family. The CAP tiling has pure-point dynamical spectrum,
  which we compute explicitly, and comes from a natural
  cut-and-project scheme with $2$-dimensional Euclidean internal
  space.  All other members of the Hat family, in particular the
  original version constructed from $30$-$60$-$90$ right triangles,
  are obtained via small modifications of the projection from this
  cut-and-project scheme.
\end{abstract}

\keywords{Tiling cohomology, Dynamical spectra, Model sets,
  Deformations, Monotile}
\subjclass[2010]{52C20, 37D40, 55N05, 52C23}

\maketitle

\section{Introduction and results}\label{sec:intro}

In March 2023, David Smith, Joseph Myers, Craig Kaplan and Chaim
Goodman-Strauss announced the construction of an aperiodic monotile
called the \emph{Hat} \cite{Hat}.  They showed that it is possible to
tile the plane with isometric copies of this tile, but only
non-periodically.  In fact, the tilings utilize $12$ tiles up to
translation: the original Hat, a reflected \emph{anti-Hat}, and
rotations of these tiles by multiples of $60$ degrees. There are
essentially two kinds of tilings that result. In one, the Hats and
anti-Hats assemble into four larger `meta-tiles', called $\tT$, $\tH$,
$\tF$ and $\tP$, which in turn form a fusion tiling.  The meta-tiles
assemble into larger `supertiles' whose adjacencies are
combinatorially the same as those of the basic meta-tiles. These
assemble into second order supertiles, which assemble into third order
supertiles, and so on to infinity.  The other construction is similar,
only using reflections of the $\tT$, $\tH$, $\tF$ and $\tP$
meta-tiles.

This result extended earlier results of Taylor \cite{Joan} and of
Socolar and Taylor \cite{ST}; see also \cite[Ex.~6.6]{TAO} for further
details. There, a functional monotile is derived from an inflation
rule with decorated hexagons that defines a unique local
indistinguishability (LI) class of tilings with \emph{perfect local
  rules} \cite[Def.~5.20]{TAO}.  Unfortunately, these local rules are
represented in the form of a decoration of a hexagon that encodes both
nearest-neighbor and next-to-nearest-neighbor information.  A purely
geometric version is possible, but only with a connected prototile
that is not simply connected, hence not disk-like. Both the decorated
hexagon and its reflected copy are needed; they occur with equal
frequency, as is nicely visible from the \emph{Llama} tiling; see
\mbox{\cite[Fig.~6.23]{TAO}}. The unique LI class defined by the
decorated hexagon is reflection symmetric, and the tiling is
limit-periodic, with pure-point spectrum, both in the dynamical and in
the diffraction sense.

The Hat achieves aperiodicity purely through geometry, with no need
for any edge markings, using a tile that is a simple non-convex
lattice polygon with $13$ edges.  Once again, the reflected version of
the Hat is also needed, but this time in such a way that reflection
symmetry is broken. This leads to \emph{two} distinct LI classes that
are mirror images of one another (an enantiomorphic pair).  That is,
if we treat the Hat and the anti-Hat as two versions of the same tile,
the local (matching) rules do not determine a single LI class, but a
pair. It is possible to force a single LI class with local rules (for
instance by declaring that two anti-Hats cannot touch), but then the
rules for the Hat and anti-Hat are different and we have a tiling by
\emph{two} tile types. In this sense, the Hat construction satisfies
the quest for a \emph{disk-like} prototile that enforces aperiodicity,
but at the price of slightly weakening the notion of local rules.

\begin{figure}
\begin{center}
\includegraphics[width=\textwidth]{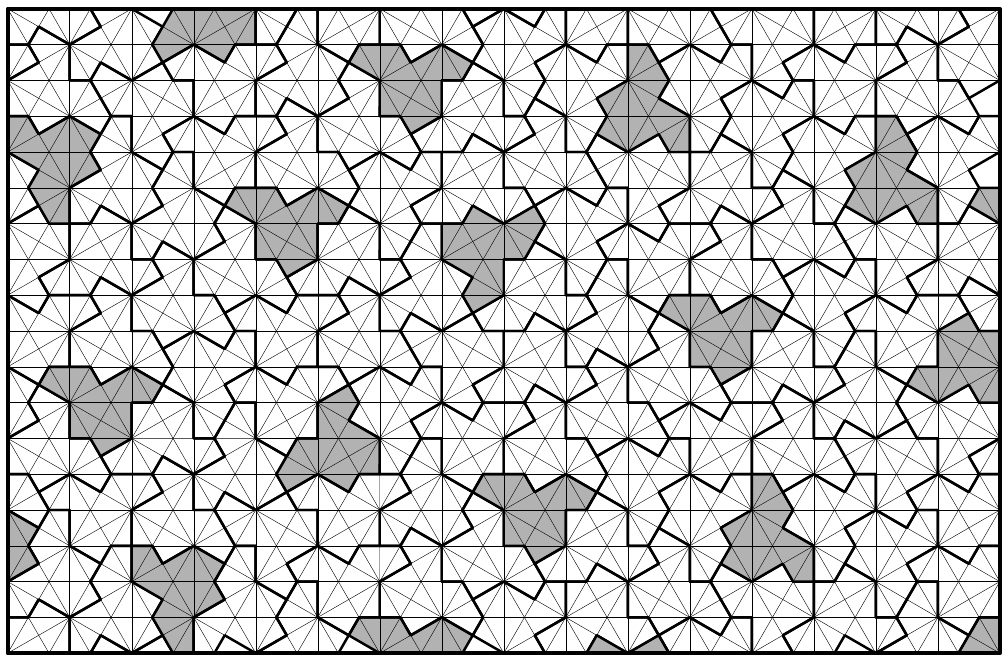}
\end{center}
\caption{\label{fig:Hatpatch}
  Patch of a Hat tiling. The anti-Hats are shaded. It has
  sixfold rotational symmetry, denoted by the cyclic group
  $C_6$, in the sense that each patch occurs in six orientations,
  but has no reflection symmetry.} 
\end{figure}

An intriguing feature of the Hat construction is that it is possible
to vary the relative lengths of the two kinds of edges in the Hat. The
geometry continues to force the tiles to assemble into meta-tiles (or
reflections of meta-tiles) with the same combinatorics as the original
Hat tiling. After rescaling to preserve the average area per tile, and
also rotating, the resultant tilings were observed to have
approximately the same large-scale structure \cite{Hat}.

Our first main result is that the shape changes considered in
\cite{Hat} are part of a $4$-dimensional family of topological
conjugacies between the translation dynamical systems of the family of
Hat tilings. For the connections between them, we need the concepts of
\emph{local derivability} and \emph{mutual local derivability} (MLD);
see \cite[Sec.~5.2]{TAO} for definitions and background.

\begin{theorem}\label{main1}
  The moduli space of shape changes to the Hat tiling, modulo MLD
  equivalence, is\/ $8$-dimensional. Four dimensions correspond to
  rigid linear transformations of\/ $\RR^2$. The other four induce
  topological conjugacies.
\end{theorem}

In particular, the translation dynamics of the $\RR^2$ action on the
hull of every version of a Hat tiling is topologically conjugate, up
to linear transformation (possibly including reflection), to the
translation dynamics on the hull of every other version of a Hat
tiling. From a dynamical systems perspective, there is only \emph{one}
Hat tiling. Note that a Hat-dominated tiling is reflected into an
anti-Hat dominated one, but since $\mathrm{GL}^{}_2(\RR)$ is not
connected, the two cannot be deformed continuously into one another.

Some of the shape changes violate the symmetries of the original
construction. The family of shape changes that respect rotational
symmetry is only $4$-dimensional and can be coded in two complex
numbers, where we identify $\RR^2$ with $\CC$. One complex number
parameterizes uniform rotations and rescalings, while the other
parameterizes shape conjugacies.  If we further insist that the shape
changes respect reflectional symmetry, the two complex parameters must
in fact be real. One represents rescalings and the other is exactly
the ratio of edge lengths described in \cite{Hat}.

Within the family of shape deformations is a \emph{self-similar}
tiling, which we call the CAP tiling (because of its relation to the
Hat and its simultaneous Cut-And-Project structure, proved
below). With this set of shape parameters, the CAP tiling can be
expressed as a \emph{geometric} inflation tiling, not only a
combinatorial one, although not one with a stone
inflation.\footnote{In a stone inflation, each super-meta-tile would
  have a footprint exactly $\phi^2$ larger than the original
  meta-tile, where $\phi = (1+\sqrt{5}\, )/2$ is the golden ratio
  (sometimes called $\tau$).} However, the CAP tiling is MLD to a
tiling by fractiles with an exact stone inflation. There are many
tools available for studying self-similar tilings, especially those
whose inflation factors are Pisot--Vijayaraghavan (PV) numbers, and we
obtain the following result.

\begin{theorem}\label{main2}
  The dynamical spectrum of the CAP tiling is pure point. Up to scale
  and rotation, as detailed in Eq.~\eqref{eq:spec} and
  Section~\textnormal{\ref{sec:compare}}, it is given by\/
  $\ZZ[\xi,\phi]$, where\/ $\xi = \exp(\pi \ii/3)$ is a primitive
  sixth root of unity. The CAP tiling itself is a cut-and-project
  tiling with internal space\/ $\RR^2$. With an appropriate choice of
  control points for the four meta-tiles, the total window is the
  hexagon shown in Figure~\textnormal{\ref{fig:window}}, with
  subdivisions for the four types.
\end{theorem}

The projection structure of the CAP tiling, respectively its control
point set, is robust under topological deformations as follows.

\begin{theorem}\label{main2a}
  Every tiling that is topologically conjugate to the CAP tiling is
  MLD to a reprojection of the CAP tiling control points. That is, it
  is MLD to a cut-and-project set with the same total $($or
  embedding$\,)$ space and the same acceptance domain as the CAP
  tiling point set, only with a different projection from the total
  space to\/ $\RR^2$. Equivalently, it is MLD to a deformed model set
  with a linear deformation function on the window.
\end{theorem}

Here, topological conjugacy refers to the existence of a homeomorphism
that commutes with the translation action of $\RR^2$ on the tilings.
Let us now turn to the topology of the continuous hull $\Omega_{\cT}$
of a Hat tiling, where $\Omega_{\cT}$ refers to the closure of the
$\RR^2$-orbit of the tiling in the standard local topology
\cite{TAO,tilingsbook}. The answer is the same for all choices of
shape, and we use the standard Hat as a representative for all of
them.

\begin{theorem}\label{main3}
  Let\/ $\cT$ be a Hat tiling and\/ $\Omega_{\cT}$ its continuous
  hull. Then, the \v{C}ech cohomology groups of\/ $\Omega_{\cT}$ are
\begin{equation*}
   \check{H}^0(\Omega_{\cT}, \ZZ)  =  \ZZ \, ,  \quad 
   \check{H}^1(\Omega_{\cT}, \ZZ)  =   \ZZ^4  ,
      \quad \text{and} \quad
   \check{H}^2(\Omega_{\cT}, \ZZ)  =  \ZZ^{10}. 
\end{equation*}
This can further be decomposed into four real representations of the
cyclic group $C_6$.  This decomposition, and the way each component
transforms under substitution, is given in
Table~\textnormal{\ref{table:coho}}.
\end{theorem}

\begin{table}
  \caption{The integer cohomology of $\Omega_\cT$ by representation
    and how it transforms under substitution. \label{table:coho}}
\renewcommand{\arraystretch}{1.2}  
\begin{tabular}{|c|c|c|c|c|}\hline 
   Representation
   &$r=1$ & $r=\xi^{\pm 1}$ & $r=\xi^{\pm 2}$ & $r=-1$ \\ \hline
   \hline
   $\check{H}^0(\Omega_{\cT},\ZZ)$ & $\ZZ$ & 0 & 0 & 0 \\ 
   Eigenvalues & 1 & && \\ \hline 
   $\check{H}^1(\Omega_{\cT},\ZZ)$ & 0 & $\ZZ^4$ & 0 & 0 \\ 
   Eigenvalues & & $\phi^{\pm 2}$ && \\ \hline 
   $\check{H}^2(\Omega_{\cT},\ZZ)$ & $\ZZ^2$ & $\ZZ^4$ & $\ZZ^2$ & $\ZZ^2$ \\ 
   Eigenvalues & $\phi^{\pm 4}$ & $\phi^{\pm 2}$ & $\xi^{\pm 1}$
     & $\phi^{\pm 2}$ \\ 
\hline
\end{tabular}
\end{table} 

The triangular lattice $\ZZ[\xi]$ appears in many places in this
analysis.  In particular, there are four special versions of the Hat,
namely the Chevron, the Hat itself, the Turtle and the Comet, in which
the vertices all lie on the lattice $c\ts \ZZ[\xi]$ for some complex
number $c$. The different values of $c$ are related by powers of
$\sqrt{5}$, $\phi - \xi$ and $\phi - \xi^5$.

\begin{remark}
  The elements of $\ZZ[\xi]$ are the Eisenstein integers, which are
  the integers in the quadratic field $\QQ (\sqrt{-3}\,)$, while the
  ring $\ZZ[\xi,\phi]$ contains integers from the number field
  $\QQ (\sqrt{-3},\sqrt{5}\,)$, but not all of them. This quartic
  field is Galois, with class number one, and has many interesting
  properties\footnote{See entry $4.0.225.1$ of $\,$
    \texttt{https://www.lmfdb.org/NumberField} $\,$ for some
    details.}, though we suppress a systematic use of them.  \exend
\end{remark}

The complex numbers $\phi-\xi^5$ and $\phi - \xi$ have norm $\sqrt{2}$
and argument $\pm \arctan \bigl(\sqrt{3/5\ts}\, \bigr)$. Additional
powers of $\phi - \xi^5$ and $\phi - \xi$ arise in the calculation of
the dynamical spectrum. As a result, the spectra of the Chevron, Hat
and Comet (but not the Turtle) are twisted relative to the lattices of
these tilings.  This twisting is different (in fact opposite) for the
tiling by meta-tiles $\tT$, $\tH$, $\tP$ and $\tF$ and the tiling by
the reflected meta-tiles. That is, the two different LI classes of the
Chevron, Hat and Comet tilings have different spectra. A similar
result holds for almost all choices of the Hat shape.

\begin{theorem}\label{thm:reflect}
  For all but two values of\/ $\alpha/\beta$, where\/ $\alpha$ and\/
  $\beta$ are the lengths of the two kinds of edges in the Hat, the
  spectra of the two LI classes of tilings built from Hats and
  anti-Hats in standard orientation are distinct.
\end{theorem}

When the spectrum of one LI class agrees with that of its reflected
copy, we know that the corresponding dynamical systems are
measure-theoretically isomorphic, by the Halmos--von Neumann theorem,
which disregards sets of measure zero. However, this does not imply
that the two systems are also topologically conjugate. In fact they
are not for these two tilings, as we explain in Remark~\ref{rem:not}.
\smallskip

The organization of this paper is as follows. In
Section~\ref{sec:deform}, we consider shape deformations of the
Hat. These are classified, up to MLD equivalence, by the vector-valued
\v{C}ech cohomology group $\check{H}^1(\Omega_{\cT}, \RR^2)$, which we
compute. We identify generators of $\check{H}^1(\Omega_{\cT}, \RR^2)$
with linear transformations and topological conjugacies, thereby
proving Theorem~\ref{main1}.  We then compute the integer-valued
\v{C}ech cohomology of $\Omega_{\cT}$, proving Theorem~\ref{main3}.

In Section~\ref{sec:SST}, we construct the self-similar CAP tiling,
whose topological dynamical system $(\Omega_{\cT}, \RR^2)$ is strictly
ergodic by standard arguments.  We show that it has pure-point
dynamical spectrum, and we compute it explicitly. We then define
control points and show how to view the resulting Delone set as a
regular cut-and-project set, proving
Theorem~\ref{main2}. Theorem~\ref{main2a} then follows from the
cohomology calculations of Section~\ref{sec:deform}.

In Section~\ref{sec:compare}, we compare several versions of the Hat
tiling, in particular the Chevron, the original Hat, the Turtle, the
Comet and the CAP. These are all topologically conjugate, up to
(complex) rescaling and rotation, where we compute the necessary
rescalings, which all involve powers of $\phi -\xi$ and
$\phi - \xi^5$.  Combined with the results of Section~\ref{sec:SST},
this yields the spectrum of the tilings obtained from each shape of
monotile, proving Theorem~\ref{thm:reflect}. We close with some
comments, observations and questions in Section~\ref{sec:final}.

\section{Shape deformations of the Hat}\label{sec:deform}

Given any tiling in $d$ dimensions, we can study changes to the shapes
and sizes of the tiles that allow the tiles to fit together in the
exact same combinatorial patterns.  Some of these changes result in
tilings that are MLD to the original one.  Infinitesimally, shape
changes modulo MLD equivalence of a tiling $\cT$ are parameterized by
$\check{H}^1(\Omega_{\cT}, \RR^d)$, where $\Omega_{\cT}$ is the orbit
closure (continuous hull) of the tiling $\cT$; see \cite{CS1}.  Since
$d=2$ and we identify $\RR^2$ with $\CC$, we must compute the
complex-valued cohomology of the hull of each tiling. There are two
such hulls, depending on whether the tiling involves meta-tiles or
reflected meta-tiles. Since these hulls are homeomorphic (being
reflections of one another), it suffices to work with a tiling by
ordinary meta-tiles.

The proof of Theorem~\ref{main1} proceeds by showing that
$\check{H}^1(\Omega_{\cT}, \CC)$ for the Hat tiling has complex
dimension $4$ (hence real dimension $8$) and that the
\emph{asymptotically negligible} subspace corresponding to topological
conjugacies has complex dimension $2$ (hence real dimension $4$). Most
of this section is devoted to computing this cohomology.
\smallskip

One standard procedure for computing the cohomology of a substitution
tiling was developed by Anderson and Putnam \cite{AP}. We construct a
branched manifold $\vG \nts$, called the Anderson--Putnam (AP)
complex, from one representative of each type of tile in the tiling,
with appropriate edge identifications. Specifically, whenever there is
a place in the tiling where tiles of type $\ft_1$ and $\ft_2$ meet
along an edge, the corresponding edges of $\ft_1$ and $\ft_2$ are
identified in $\vG\nts$.  Substitution induces a map from this
complex to itself. As long as the substitution meets a technical
condition called \emph{forcing the border} \cite{kel} (defined below),
the tiling space is homeomorphic to the inverse limit of the AP
complex under substitution, and the \v{C}ech cohomology
$\check H^k(\Omega_\cT)$ of the tiling space is the direct limit of
the (ordinary) cohomology $H^k(\vG\nts)$ of the AP complex.

Anderson and Putnam developed their machinery for self-similar
tilings, but the \emph{topology} of the tiling space does not depend
on the \emph{geometry} of the tiles.  A nearly identical procedure
applies to combinatorial substitution systems like the Hat, or more
generally to fusion tilings of any kind; see \cite[Sec.~5]{FS} for
details. The $n$th instance of the AP complex is replaced by a complex
constructed from $n$th order supertiles identified along supertile
edges and the substitution map (between a space and itself) is
replaced by a \emph{forgetful map} from a complex built from $n$th
order supertiles to a complex build from $(n{\ts -}1)$st order
supertiles.

As long as the construction of $n$th order supertiles from
$(n{\ts -}1)$st order supertiles is combinatorially the same for every
$n$, the complexes are all homeomorphic and have the same cohomology,
and the forgetful maps induce the same maps on cohomology.  Just as
with the (naive) Anderson--Putnam procedure, the tiling space is the
inverse limit of these complexes and the \v{C}ech cohomology
$\check H^k(\Omega_\cT)$ of the tiling space is the direct limit of
the cohomology $H^k(\vG\nts)$ of the complexes.

A substitution (or fusion) tiling is said to \emph{force the border}
\cite{kel} if, for some $k$, every $k$th order supertile enforces a
patch of tiles completely covering the supertile plus an extra margin
of positive thickness around the supertile. If such a supertile is
further inflated, this extra margin, where the tiling is also
determined, grows in thickness at the same rate as the supertile. The
result is that a single infinite-order supertile, if it does not
already cover the whole plane, has a unique legal completion to the
whole plane.  It is this latter property we are interested in, as it
simplifies the construction of the tiling space as an inverse limit
space, and hence the computation of its \v{C}ech cohomology.

\begin{figure}
\begin{center}
\includegraphics[width=0.6\textwidth]{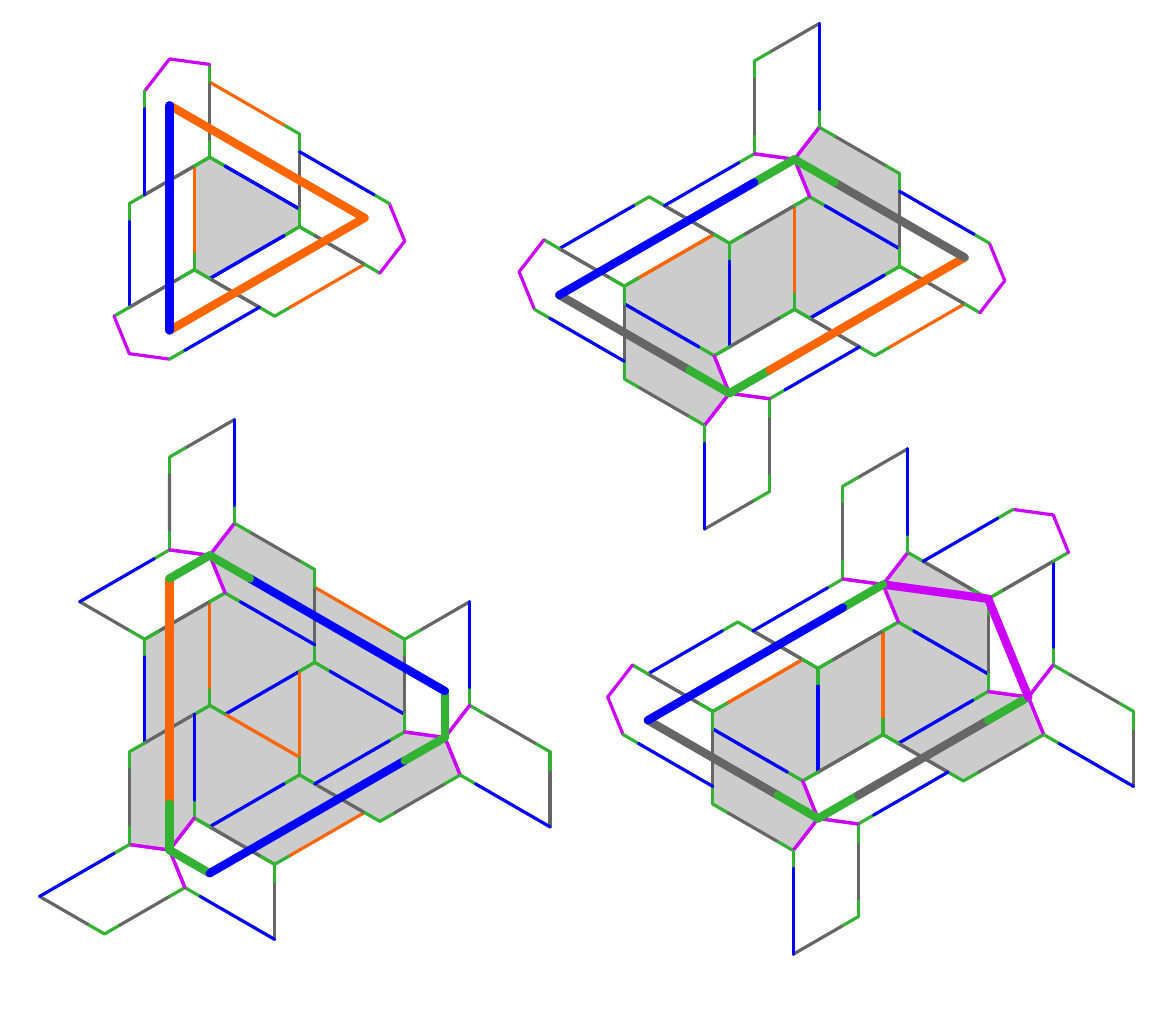}
\end{center}
\caption{\label{fig:infp}Patches of tiles enforced by the level-one
  supertiles. The tiles that belong to a supertile are shaded.  These
  patches extend beyond the outline of the supertiles (thick lines) by
  a margin of positive thickness, thus proving the border-forcing
  property. There is some choice in the assignment of tiles to
  supertiles. Here, we use the original one from \cite{Hat}, while
  other choices are also possible, then leading to different
  fractalizations of tile edges (in comparison to
  Figure~\ref{fig:frac}).}
\end{figure}

\begin{lem}\label{lem:BFnew}
  The substitution on the meta-tiles\/ $\tT$, $\tH$, $\tP$ and\/ $\tF$
  forces the border.
\end{lem}
\begin{proof}
  In Figure~\ref{fig:infp}, we show patches of tiles which are
  enforced by the first order supertiles. The tiles intersecting the
  interior of their supertile were already given in
  \cite[Figure~2.8]{Hat}. We had to complement these patches only by a
  few tiles having a vertex on the supertile boundary. These extra
  vertices are either a purple tip of a tile or a supertile, where a
  local three-fold symmetry is required by the matching conditions. In
  the terminology of \cite{TAO}, we replace the inflation rule by a
  consistent pseudo-inflation with a complete `belt' of positive
  thickness. These pictures show that the enforced patches extend by a
  positive margin beyond the perimeter of the supertile. Under
  inflation, the thickness of this margin, where the tiling is also
  determined, grows at the same rate as the supertile size, which
  implies border-forcing.
\end{proof}

Since the substitution forces the border, we can use the uncollared AP
complex to compute the cohomology \cite{AP}. We work at the level of
(meta-)tiles, but the construction for supertiles of any order is
combinatorially identical.

\begin{figure}[t]
\begin{center}
\includegraphics[width=0.6\textwidth]{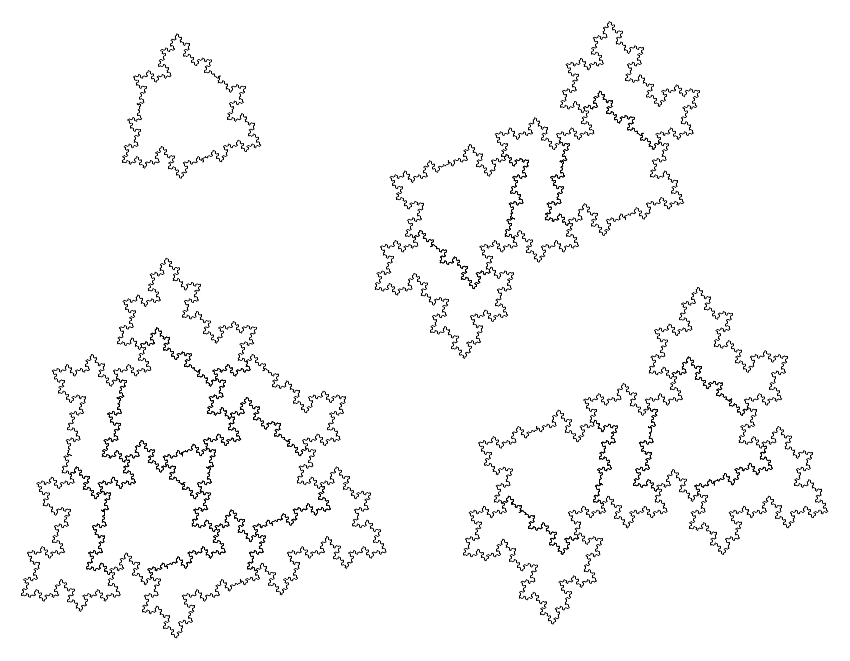}
\end{center}
\caption{\label{fig:frac}First level supertiles with fractalized
  edges. These are obtained by iterating the inflation on a tile and
  rescaling the outline of the resulting patch back to the original
  size. For these \emph{fractiles}, we have a stone inflation.  The
  tiling obtained from this inflation is MLD with the CAP tiling from
  Figure~\ref{fig:patch}. }
\end{figure}

\begin{figure}[hb]
\begin{center}
\includegraphics[width=0.7\textwidth]{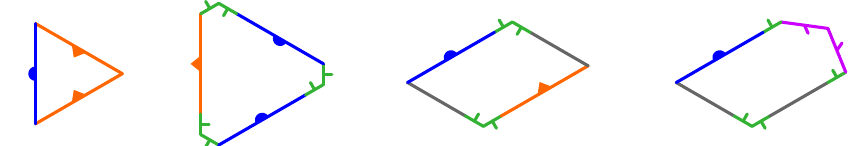}
\end{center}
\caption{\label{fig:tiles2} The four meta-tiles in reference
  orientation. We have added markers to the edges as in \cite{Hat} to
  fix the positive direction of each edge (except for the $L$ edge,
  which is non-directional, see below). In the positive direction, the
  marker is always on the left side of the edge.}
\end{figure}

In the Hat tiling, there are four kinds of meta-tiles, called $\tT$,
$\tH$, $\tP$ and $\tF$, as shown in Figure~\ref{fig:tiles2}.  Each
appears in six orientations, where we choose the reference orientation
to be the one shown in Figure~\ref{fig:tiles2}. That is, our tiles are
$r^k \tT$, $r^k \tH$, $r^k \tP$ and $r^k \tF$, where $k$ ranges from
$0$ to $5$ and $r$ means `positive rotation by $60$ degrees'. The
inflations of these tiles can be read off from Figure~\ref{fig:infp}.

\begin{figure}
\begin{center}
\includegraphics[width=0.6\textwidth]{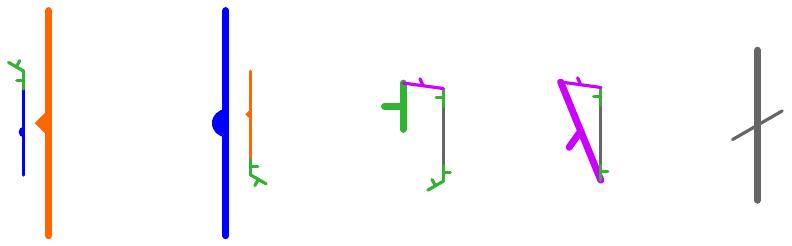}
\end{center}
\caption{\label{fig:edges} The five types of (oriented) edges $A$,
  $B$, $X$, $\fe$ and $L$ in reference orientation, along with their
  inflations. In the positive edge direction, the marker is on the
  left (except for the $L$ edge, which is non-directional). The edge
  inflations can in principle be extracted from Figure~\ref{fig:infp},
  but the determination of edge orientations is a bit subtle.  }
\end{figure}

There are five kinds of edges, labeled $A$, $B$, $X$, $\fe$, and $L$,
all appearing in Figure~\ref{fig:edges}.  The first four edges have a
clear direction. We define the reference orientation to be vertical
(or roughly vertical for $\fe$), with the relevant marking on the
left. The $L$ edge is not directional, so as a chain we have
$r^3 L = -L$.

Now, we turn to vertices. Let $A_+$ and $A_-$ denote the endpoints and
beginning points of an $A$ edge, and likewise with $B$, $X$, $\fe$ and
$L$. The symmetry of the $L$ edge implies that $L_+ = r^3 L_-$. Other
relations are obtained from Figure~\ref{fig:tiles2} by identifying the
end of one edge with the beginning of the next edge on the boundary of
a meta-tile. This yields the following relations:
\begin{align}
   A_+ & \, = \,  r^5 X_-, &  A_- & \, = \, r^3 X_-, \\ 
   B_+ & \, = \,  X_-, & B_- & \, = \, r^2 X_-, \\ 
   X_+ & \, = \,  X_+, &  X_- & \, = \, X_-, \\ 
   \fe_+ & \, = \,  X_+, &  \fe_- & \, = \, \fe_-, \\ 
   L_+ & \, = \,  X_-, &  L_- & \, = \, r^3 X_-, 
\end{align}
expressing every vertex in terms of $X_-$, $\fe_-$ and
$X_+$. Furthermore, 
\begin{equation}
  \fe_- \, = \, r^2 \fe_- \, = \, r^4 \fe_- \ts , \qquad\qquad
  X_+ \, = \, r^2 X_+ \, = \, r^4 X_+ \ts . 
\end{equation}

Having identified the faces, edges and vertices of the AP complex, we
now compute the boundary maps, relative to the bases
$\{\tT, \tH, \tP, \tF \}$ for faces, $\{ A, B, X, \fe, L\}$ for edges
and $\{X_-, X_+, \fe_-\}$ for vertices.

\begin{equation}
  \partial^{}_1 \, = \, \begin{pmatrix} 
  r^5{-}\ts r^3 & 1{-}\ts r^2 & -1 & 0 & 1{-}\ts r^3 \\ 
  0&0&1&1&0 \\  0&0&0&-1&0  \end{pmatrix}
\end{equation}

\begin{equation}
  \partial^{}_2 \, = \, \begin{pmatrix}
   r {+}\ts  r^5 & -1 & r^5 & 0 \\ 
   -1 & r {+}\ts r^{5} & -r^5 & -r^5 \\ 
   0&0& (r{-}\ts r^5)(1{+}\ts r^3) & r^4{-}\ts r^2 \\ 
   0 & 0 & 0 & r\ts {-}\ts r^3 \\ 
   0 & 0 & 0 & r^4 {+} \ts  r^5
\end{pmatrix}.
\end{equation}

Here, we view chains as column vectors, to be acted on from the left
by $\partial^{}_1$ and $\partial^{}_2$. Co-chains are row vectors, to
be acted on from the right.  Strictly speaking, each entry should be
viewed as a $6 \times 6$ block, with `1' being the identity matrix and
\[
  r \, = \, \begin{pmatrix} 0 & 0 & 0 & 0 & 0& 1 \\ 
  1 & 0 & 0 & 0 & 0 & 0 \\ 
  0 & 1 & 0 & 0 & 0 & 0 \\ 
  0 & 0 & 1 & 0 & 0 & 0 \\
  0 & 0 & 0 & 1 & 0 & 0 \\ 
  0 & 0 & 0 & 0 & 1 & 0 \end{pmatrix},
\] 
and we further need to apply the identities $L=-r^3L$, $X_+=r^2 X_+$
and $\fe_-=r^2\fe_-$. Fortunately, it is not necessary to work with
gigantic matrices. Instead, we follow the methods of \cite{ORS}, see
also \cite[Ch.~4]{tilingsbook}, and work with one representation of
$C_6$ at a time. The total cohomology is then the direct sum 
of the contributions of each representation.

Working over the complex numbers, a representation is just a choice of
a $6$th root of unity to assign to $r$. That is, we must consider
$r=\xi^k$ where $k \in \{0, 1, 2, 3, 4, 5\}$.  However, the rows
corresponding to $X_+$ and $\fe_-$ only appear in the representations
where $r^2=1$, and the row and column corresponding to $L$ only appear
in the representations where $r^3=-1$.

\begin{remark}
  Computing the real-valued and integer-valued cohomology is somewhat
  more complicated, both because the irreducible \emph{real}
  representations of $C_6$ are not all $1$-dimensional and because,
  for integer-valued cohomology, decomposing into representations is
  only guaranteed to compute a finite-index subgroup of the
  cohomology. We will compute the integer-valued cohomology of the Hat
  tiling later in this section.  \exend
\end{remark}

The sizes and ranks of the matrices $\partial^{}_1$ and
$\partial^{}_2$ in each representation are shown in Table
\ref{table:ranks}, together with the contributions of each
representation to $H^0(\vG,\CC)$, $H^1(\vG,\CC)$ and
$H^2(\vG,\CC)$.  \smallskip

\begin{table}[t]
  \caption{The contribution of each representation of $C_6$ to the
    complex cohomology $H^k(\vG, \CC)$ of the Anderson--Putnam
    complex \label{table:ranks}}
\renewcommand{\arraystretch}{1.5}  
\begin{tabular}{|c|c|c|c|c|c|c|}\hline
Representation & $r=1$ & $r=\xi$ & $r=\xi^2$ & $r=\xi^3$
  & $r=\xi^4$ & $r=\xi^5$ \\ 
\hline \hline
Number of faces & 4 & 4 & 4 & 4 & 4 & 4 \\ \hline
Number of edges & 4 & 5 & 4 & 5 & 4 & 5 \\ \hline
Number of vertices & 3 & 1 & 1 & 3 & 1 & 1 \\ \hline
Rank of $\partial^{}_2$ & 2 & 2 & 3 & 2 & 3 & 2 \\ \hline
Rank of $\partial^{}_1$ & 2 & 1 & 1 & 3 & 1 & 1 \\ \hline
Contribution to $H^0$ & $\CC$ & 0 & 0 & 0 & 0 & 0 \\ \hline
Contribution to $H^1$ & 0 & $\CC^2$ & 0 & 0 & 0 & $\CC^2$ \\ \hline
Contribution to $H^2$ & $\CC^2$ & $\CC^2$ & $\CC$
     & $\CC^2$ & $\CC$ & $\CC^2$ \\ \hline
\end{tabular}
\end{table} 

Having computed $H^1(\vG,\CC)$, we must take a direct limit
under the pullback of the forgetful map to obtain the \v{C}ech
cohomology $\check H^1(\Omega_\cT,\CC)$ of the hull
$\Omega_{\cT}$. There are manifestly two generators of
$H^1(\vG, \CC)$ that transform with eigenvalue $\phi^2$ under
substitution. This is because the cochains that assign to each edge
its horizontal or vertical displacement get stretched asymptotically
by $\phi^2$ under substitution. Since substitution acts on
$1$-cochains, and therefore on $H^1(\vG,\CC)$, via a matrix with
integer coefficients, the algebraic conjugate of $\phi^2$, namely
$\phi^{-2}$, must also be an eigenvalue with multiplicity two. Since
$H^1(\vG,\CC)$ is only $4$-dimensional, substitution acts
invertibly on $H^1(\vG,\CC)$ and the direct limit of
$H^1(\vG,\CC)=\CC^4$ is just $\CC^4$. We conclude that
$\check{H}^1(\Omega_{\cT}, \CC) = \CC^4 = \RR^8$, with a subspace of
real dimension $4$ that expands under substitution with eigenvalue
$\phi^2$ and a subspace of real dimension $4$ that contracts under
substitution with eigenvalue $\phi^{-2}$.

The set of asymptotically negligible classes, which describe shape
changes (modulo MLD) that are topological conjugacies, is exactly the
contracting subspace under substitution \cite{CS1}.  Thus, the
$8$-dimensional family of shape deformations of the Hat (modulo MLD)
breaks into a $4$-dimensional family of changes to the large-scale
structure and a $4$-dimensional family of shape conjugacies.

Since we can transform any tiling by applying a rigid linear
transformation, thereby changing its large-scale structure, and since
the space of linear transformations is $4$-dimensional, the shape
classes of any two Hat tilings have expansive components that are
related by such a linear transformation. After applying a linear
transformation to one of them, the shape classes differ by something
asymptotically negligible, so the tiling spaces are topologically
conjugate.

So far, the analysis has only applied to tilings by meta-tiles, not by
anti-meta-tiles.  However, tilings by anti-meta tiles are already
related by a linear transformation (namely reflection) to tilings by
standard meta-tiles. Thus, \emph{every} space of tilings by Hats,
whether involving meta-tiles or anti-meta-tiles, is topologically
conjugate, up to linear transformation, to every other such
space.\footnote{Note that a Hat tiling is homeomorphic, but not
  topologically conjugate, to an anti-Hat tiling.}  Likewise, this
applies also to the shape changes that are used to define Socolar's
Key tiles \cite{Soc}.

This concludes the proof of Theorem~\ref{main1}.

\begin{remark}
  Both the expansive classes and the asymptotically negligible classes
  come equally from the $r=\xi$ and $r=\bar{\xi}=\xi^5$
  representations. If we want to preserve rotational symmetry, the
  vector along the $r^k E$ edge (where $E$ is any of
  $\{ A, B, X, \fe, L\}$) should be $\xi^k$ times the vector along the
  $E$ edge. That is, shapes that respect rotational symmetry live
  entirely in the $r=\xi$ representation, while contributions from the
  $r=\xi^5$ representation break rotational symmetry.  \exend
\end{remark}

\medskip

A (relatively) simple computation of $\check{H}^1(\Omega_{\cT}, \CC)$
was sufficient to prove Theorem~\ref{main1}.  However, a more complete
computation of the cohomology of $\Omega_{\cT}$, and how it transforms
under rotation and substitution, is also of interest, yielding insight
into the structure of Hat tilings.

First, we consider real-valued cohomology. For any topological space
$X$, any cohomology theory, and any degree $k$, the real dimension of
$H^k(X,\RR)$ is the same as the complex dimension of
$H^k(X,\CC)$. This implies that
$\check H^1(\Omega_{\cT},\RR) = H^1(\vG,\RR) = \RR^4$ and that
$\check H^2(\Omega_{\cT},\RR) = H^2(\vG,\RR) = \RR^{10}$. The
decomposition into representations is slightly more subtle, since the
complex representations with $r$ being $\xi$, $\xi^2$, $\xi^4$ or
$\xi^5$ are not complexifications of real representations.  Rather,
the direct sum of the $r = \xi^{\pm 1}$ representations is the
complexification of a $2$-dimensional real representation, as is the
direct sum of the $r = \xi^{\pm 2}$ representations.

Rather than simply relying on stretching and algebraic conjugacy, we
compute directly how substitution acts on $H^1(\vG,\RR)$ and
$H^2(\vG,\RR)$.  The action of substitution on faces can be read
off from Figure~\ref{fig:tiles} and the action on edges and vertices
from Figure~\ref{fig:edges}. The matrices for the action on vertices,
edges, and faces are:
\begin{eqnarray}
\sigma^{}_0 & \, = \, & \begin{pmatrix}
   r^5 & 0 & 0 \\ 0 & 0 & r^3 \\ 0 & r^4 & 0 \end{pmatrix}, \\[1mm]
\sigma^{}_1 & \, = \, & \begin{pmatrix} 
   0 & 1 & 0 & 0 & 0 \\ 
   1 & 0 & 0 & 0 & 0 \\ 
   1{-}\ts r^4 & r-r^3 & 1{+}\ts\ts r^5{-}\ts r^3
   & 1{-}\ts r^3 & 0 \\ 
   0 & 0 & -r^4 & -r^4 & 0 \\ 
   0 & 0 & 1 & 1 & r^5 \end{pmatrix}, \\[1mm]
\sigma^{}_2 & \, = \, &  \begin{pmatrix} 
   0 & r^5       & 0         & 0  \\ 
   1 & 2{+}\ts r^4     & 1{+}\ts r^5     & 1{+}\ts r^5 \\  
   0 & 1{+}\ts r{+}\ts r^2   & r         & r \\ 
   0 & 1{+}\ts r^2{+}\ts r^4 & r^2{+}\ts r^5  & 1{+}\ts r^2{+}\ts r^5
  \end{pmatrix}. 
\end{eqnarray}

Acting on $H^2(\vG, \RR)$, we can study this one representation
at a time.  The substitution eigenvalues are exactly those listed in
Table~\ref{table:coho}. Note that the product of these eigenvalues is
1, so substitution acts via a unimodular matrix on
$H^2(\vG, \RR)$.

We next compute $H^*(\Omega_{\cT}, \ZZ)$ more carefully, without
decomposing by representation, using the full $27 \times 24$ matrix
$\partial^{}_2$ and the full $10 \times 27$ matrix
$\partial^{}_1$. The result is that $H^1(\vG, \ZZ)=\ZZ^4$ and
$H^2(\vG, \ZZ)=\ZZ^{10}$.  Since substitution acts via
unimodular matrices, the direct limits of $\ZZ^4$ and $\ZZ^{10}$ are
simply $\ZZ^4$ and $\ZZ^{10}$. As with the real-valued cohomology,
these can be decomposed into direct sums of modules of the form
$\ZZ[r]/(r-1)$ (corresponding to $r=1$), $\ZZ[r]/(r+1)$ (corresponding
to $r=-1$), $\ZZ[r]/(r^2-r+1)$ (corresponding to $r=\xi^{\pm 1}$) and
$\ZZ[r]/(r^2+r+1)$ (corresponding to $r = \xi^{\pm 2}$). This
establishes Theorem \ref{main3}.

\medskip

Finally, we turn to a closer examination of
$\check{H}^1(\Omega_{\cT}, \CC) \simeq H^1(\vG,\CC)$, which has
components in the $r = \xi^{\pm 1}$ representations. In the $r=\xi$
representation, the left-kernel of the matrix $\partial^{}_2$ is
spanned by the three row vectors
\begin{equation}
  \label{eq:threemu}
  \begin{split}
  \bmu^{}_1 & \, = \,  (1+\xi^5, 1+\xi^5, 1, 0 ,0) \ts , \\ 
  \bmu^{}_2 & \, = \,  (\xi + \xi^2, \xi + \xi^2, 0, 1, 0) \ts , \\ 
  \bmu^{}_3 & \, = \,  (1+\xi^5, 1+\xi^5, 0,0,1) \ts .
\end{split}
\end{equation}
The row space of $\partial^{}_1$ is spanned by
$2 \bmu^{}_3 - \bmu^{}_1$, so $\bmu^{}_1$ is cohomologous to
$2 \bmu^{}_3$.

Under right-multiplication by $\sigma^{}_1$, $\bmu^{}_1$ transforms to
$(2+\xi^5) \bmu^{}_1 + 2 \bmu^{}_2$, while $\bmu^{}_2$ transforms to
$\xi (\bmu^{}_1 + \bmu^{}_2)$. That is, the action of substitution on
$H^1(\vG,\CC)$ in the $\{ \bmu^{}_1, \bmu^{}_2 \}$ basis and in
the $r=\xi$ representation is via the matrix
\begin{equation} 
  \widetilde{\sigma}^{}_1 \, = \, \begin{pmatrix}
   2 {+}\ts  \xi^5 & 2 \\ \xi & \xi  \end{pmatrix},
\end{equation}
where as usual we are thinking of cohomology classes as rows with the
matrix acting to the right. The left and right eigenvectors
corresponding to the eigenvalues $\lambda = \phi^{\pm 2}$ take the
form
\begin{equation}\label{eq:eigen}
  \ell^{}_\lambda \, = \, (\lambda - \xi, 2), \qquad 
  r^{}_\lambda \, = \, \begin{pmatrix}
     \lambda - \xi \\ \xi \end{pmatrix}. 
\end{equation}
Results for $r=\xi^5$ are complex conjugates of the corresponding
results for $r=\xi$.

\begin{figure}
\begin{center}
\includegraphics[width=0.6\textwidth]{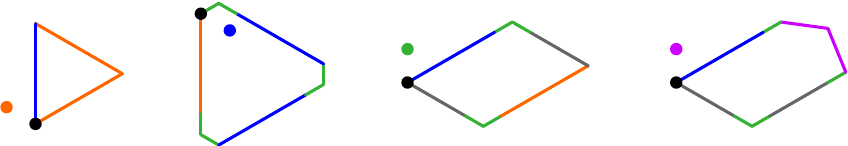}
\end{center}
\caption{\label{fig:tiles}The tiles of the CAP tiling, together with
  their control points (orange, blue, green and purple for $\tT$-,
  $\tH$-, $\tP$- and $\tF$-tiles, respectively). Relative to the
  corners marked with black dots, the control points are at positions
  $\ii\phi\xi$, $-\ii\phi\xi$, $\ii\phi$ and $\ii\phi$, respectively.
  Apart from the $\tH$-tiles, the control points are outside their
  tile, but never coincide with a control point of a neighboring tile.
  Each tile occurs in six different orientations, but not in a
  reflected version.  }
\end{figure}

\section{The CAP tiling}\label{sec:SST}

In this section, we construct the CAP tiling. We work directly with
the shape of the $\tT$, $\tH$, $\tP$ and $\tF$ meta-tiles, or
equivalently the vector displacements along the $A$, $B$, $X$, $\fe$
and $L$ edges. We choose
\[
  (A, B, X, \fe, L) \, = \,
  (3\phi, 3\phi, 1, \phi+\xi, 2\phi) \, = \,
  \bmu^{}_1 + (\phi+\xi)\bmu^{}_2 + 2\phi \bmu^{}_3 \ts .
\]
This is cohomologous to $(\phi+1)\bmu^{}_1 + (\phi+\xi)\bmu^{}_2$,
which we represent as $(\phi+1, \phi+\xi)$. This is a left-eigenvector
of $\widetilde{\sigma}^{}_1$ with eigenvalue $\phi^2$.  The
substitution rule (shown in Figure~\ref{fig:infp}) is self-similar in
the sense that a substituted tiling (with shape class
$(\phi+1, \phi+\xi) \widetilde{\sigma}_1$) is then MLD to a rescaling
of the original tiling by $\phi^2$ (with shape class
$\phi^2(\phi+1, \phi+\xi)$). This constitutes the LIDS property in the
sense of \cite[Def.~5.16]{TAO}.

\begin{figure}
\begin{center}
\includegraphics[width=0.9\textwidth]{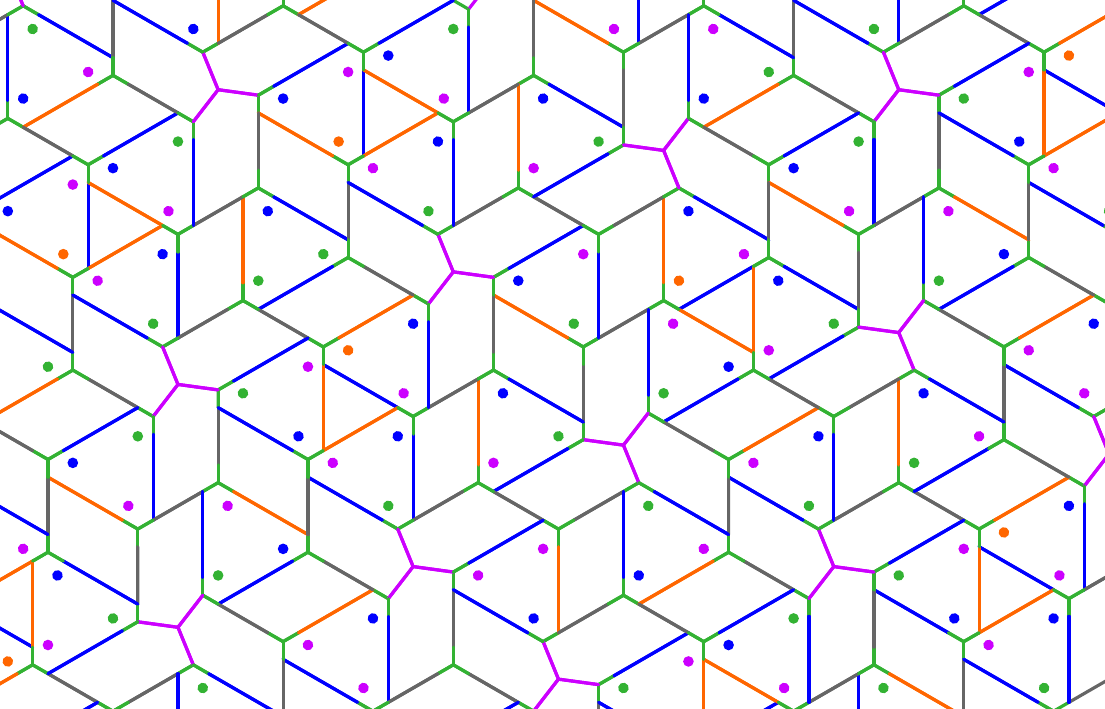}
\end{center}
\caption{\label{fig:patch}A patch of the CAP tiling. All control
  points occur in triples inside the $\tH$-tiles. One point (blue) of
  such a triple comes from the $\tH$-tile itself, the other two from
  the two tiles adjacent along one of the two blue edges. The
  distances between any control points are elements of the return
  module $R$.  }
\end{figure}

The tiles with these shape parameters are shown in
Figure~\ref{fig:tiles}, where we have also added a control point for
each tile. A patch of the resulting tiling, also including control
points, is shown in Figure~\ref{fig:patch}. The control points have
been chosen such that they lie in a single orbit of the \emph{return
  module} $R$ of the tiling.  The return module is the $\ZZ$-span of
all \emph{return vectors}, which in turn are vectors which translate a
finite patch of a tiling to an identical copy of that patch in the
same tiling.

Control points of tiles of the same type and orientation are clearly
in a single $R$-orbit, but if this is to hold for all control points,
a number of constraints must be satisfied. In Figure~\ref{fig:patch},
one can check that
\begin{equation}\label{eq:t-vec}
   t^{}_{0} \, = \, 3\phi+2-\xi=\phi^2(1+\xi)(\phi-\xi)
\end{equation}
is a return vector, and that all other return vectors are contained in
the module generated by $\phi \ts t^{}_{0}$, $t^{}_{0}$,
$\phi\ts \xi t^{}_{0}$, and $\xi t^{}_{0}$. This module is
$t^{}_{0} \ZZ[\xi,\phi]$, which is invariant under inflation, and is
indeed our return module $R$ (see also Section~\ref{sec:compare}). In
Figure~\ref{fig:patch}, one can also convince oneself that all control
points are separated by vectors in $R$.

The return module $R$ is needed to determine the point spectrum of the
CAP tiling, but also to represent its set of control points as a
cut-and-project set. The latter is possible only if the dynamical
spectrum is pure point, which we show next. A classical method to
prove this for self-similar inflation tilings is Solomyak's overlap
algorithm \cite{Sol97}, which works as follows. Suppose $\cT$ is an
instance of such an inflation tiling. Solomyak showed that its
spectrum is pure point if and only if, for every return vector $t$,
\[
  \coinc \bigl( \sigma^n(\cT),\sigma^n(\cT-t) \bigr)
  \, \xrightarrow{\, n\rightarrow\infty \,} \, 1 \ts ,
\]
where $\sigma$ is the inflation map, and $\coinc (\cT_1,\cT_2)$ is the
area fraction covered by coincident tiles of the two tilings $\cT_1$
and $\cT_2$. It is actually enough to show this for $t$ in a basis of
the return module. The area not covered by coincident tiles is called
the \emph{discrepancy} area.

The space can now be cut into compact pieces, called \emph{overlaps},
which are just the (non-empty) overlaps of all pairs of tiles, one
from each tiling. Overlaps formed by two coincident tiles (coincidence
overlaps) cover the coincidence area, whereas the remaining
discrepancy overlaps cover the discrepancy area. There is an inflation
defined on overlaps: one inflates the two tiles, and decomposes the
overlap of the two supertiles into overlaps. Due to finite local
complexity, for any fixed return vector $t$, the number of different
overlap types (up to translation) is finite, and a PV inflation factor
guarantees that only a finite number of overlap types is generated
under inflation.

For a stone inflation, the corresponding overlap inflation can now be
used to test Solomyak's criterion. The area fraction covered by
coincident tiles converges to $1$ if and only if every overlap type
eventually produces a coincidence overlap. In this case, the area
fraction covered by the discrepancy overlaps shrinks at each inflation
step by a certain factor bounded away from $1$. Conversely, if at
least one overlap type does not produce any coincidence, there is a
discrepancy region whose area fraction does not shrink.

Here, we have a stone inflation only for the fractiles, for which it
is difficult to determine whether two tiles have a non-empty overlap.
For this situation, there is a refined algorithm by Akiyama and Lee
\cite{AL}. It works with potential overlaps in a first stage, and
eliminates those which are not true overlaps at a later stage. This
refined algorithm shows that the CAP tiling has pure point
spectrum. Unfortunately, the details are complicated and had to be
done by computer. For readers who find computer-assisted proofs
unsatisfying, we give an alternative proof of pure-point spectrum
below.

\begin{lem}\label{lem:pp}
  The topological dynamical system\/ $(\Omega_{\cT}, \RR^2)$ induced
  by the CAP tiling is strictly ergodic. Further, it has pure-point
  dynamical spectrum, where all eigenfunctions have continuous
  representatives.
\end{lem}

\begin{figure}
\centerline{\includegraphics[width=0.9\textwidth]{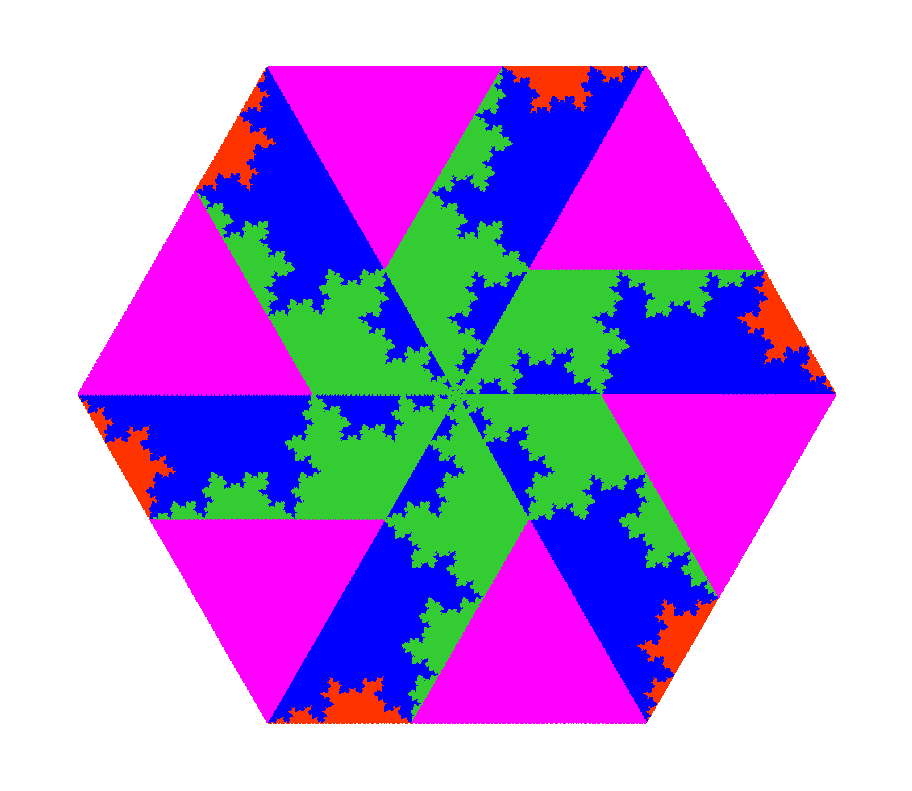}}
\caption{\label{fig:window}Window for the cut-and-project scheme of
  the CAP tiling. The color code refers to the type of the tile whose
  control point is plotted (orange for $\tT$-, blue for $\tH$-, green
  for $\tP$- and purple for $\tF$-tiles). The subdivision for the
  different tile types is partly fractal, see Eq.~\eqref{eq:dim}, and
  partly straight. The edge length of the window is $\phi$ times the
  unit length in $R$.  The image is \emph{chiral} with sixfold
  rotational symmetry, but has no reflection symmetry, although the
  total window, and hence the set of all control points (without
  color), is mirror symmetric.}
\end{figure}

Now, we show that the set $\vL$ of control points of the CAP tiling,
which is a subset of the return module $R$, forms a cut-and-project
set (CPS).  For this, we need the construction from \cite{BM} which
gives the CPS from intrinsic data of the tiling; see also
\cite[Sec.~5.2]{Nicu} for a detailed description. The procedure
simplifies for inflation point sets, where one can more easily
determine the limit translation module from the return vectors. The
upshot is that the limit translation module agrees with our return
module, and needs to be lifted to an embedding lattice.  Due to the
underlying structure with the PV unit $\phi^2$, one natural way to
turn $R$ into a lattice in $\RR^4$ employs the classic Minkowski
embedding; see \cite[Sec.~3.4]{TAO}.

To this end, we first rescale $\vL$ by a factor $t^{-1}_{0}$, see
Eq.~\eqref{eq:t-vec}, so that its return module is equal to
$\ZZ[\xi,\phi]$. This subset is then easily lifted to the Minkowski
embedding of $\ZZ[\xi,\phi]$, and then projected to internal
space. The Minkowski embedding $\cL$ of $\ZZ[\xi,\phi]$ is obtained by
combining the basis vectors $\{\phi,1,\xi\phi,\xi\}\subset\CC$ with
their Galois conjugates, so that we obtain the basis vectors of $\cL$
as
$\{(\phi,1-\phi),(1,1),(\phi\ts
\xi,(1-\phi)\xi^5),(\xi,\xi^5)\}\subset\CC^2$.  The resulting window
is shown in Figure~\ref{fig:window}. This is justified by the general
uniform distribution result for this setting \cite{Moody}.

\begin{remark}
  The choice of control points is not unique, and any other
  representative from the same MLD class can be used as well.  Matters
  are simplified by using only \emph{one} translation class with
  respect to $\ZZ[\xi,\phi]$, as one has to use several windows for
  the different translation classes otherwise (as also occurs in
  \cite{Soc}). Even so, other choices are possible with simple total
  window, for instance one where the center of the outer hexagon is a
  smaller hexagon for all control points of $\tF$-tiles, surrounded by
  a belt of six trapezoids that are fractally subdivided into the
  other three types, which then breaks the reflection symmetry. This
  choice would simply amount to a re-coloring of the control points we
  use here. \exend
\end{remark}

Put differently, we are using the following \emph{cut-and-project
    scheme} \cite[Sec.~7.2]{TAO}
\begin{equation}\label{eq:CPS}
\renewcommand{\arraystretch}{1.2}\begin{array}{r@{}ccccc@{}l}
   \\  & \RR^2 & \xleftarrow{\;\;\; \pi \;\;\; }
   & \RR^2 {\times}\ts \RR^2 &
   \xrightarrow{\;\: \pi^{}_{\text{int}} \;\: } & \RR^2 & \\
   & \cup & & \cup & & \cup & \hspace*{-1.5ex}
   \raisebox{1pt}{\text{\footnotesize dense}} \\
   & \pi (\cL) & \xleftarrow{\;\ts 1-1 \;\ts } & \cL &
     \xrightarrow{ \qquad } &\pi^{}_{\text{int}} (\cL) & \\
   & \| & & & & \| & \\
   & \ZZ[\xi,\phi] & \multicolumn{3}{c}{\xrightarrow{\qquad\quad\quad
    \,\,\,\star\!\! \qquad\quad\qquad}}
   &  \ZZ[\xi,\phi]  & \\ \\
\end{array}\renewcommand{\arraystretch}{1}
\end{equation}
with the natural projections $\pi$ and $\pi^{}_{\mathrm{int}}$ to
direct (or physical) and internal (or perpendicular) space. The
$\star$-map is the Galois automorphism induced by the joint action of
$\phi \mapsto \phi^{\ts \prime} = 1-\phi$ and
$\xi \mapsto \bar{\xi} = \xi^5$. This is the mapping used for the
Minkowski embedding. It is not the only choice, but any other one is
equivalent to it via standard modifications.

The point density of our CPS with the above lattice and window is
given by the density of the lattice times the area of the total
window. The lattice $\cL$ has a unit cell of volume $15/4$, so we get
$\dens (\cL) = 4/15$.  Indeed, each summand $\ZZ[\phi]$ provides a
factor $\sqrt{5}$ (as known from the Fibonacci tiling), and because
the sum is not orthogonal, we get two extra factors $\sqrt{3}/2$, as
known from the triangular lattice. The window is a regular hexagon of
edge length $\phi$, hence has area $\frac{3}{2} \phi^2 \sqrt{3}$, so
that the density of the covering model set with this window becomes
\[
      \rho^{}_1 \, = \, \myfrac{2}{5} \sqrt{3}\,\phi^2.
\]

Let us now describe in more detail the cut-and-project setting in
terms of the inflation. For this purpose, denote by
$\ft_1,\ldots, \ft_{24}$ the $24$ translational prototiles, which
emerge from the $6$ orientations of the tiles $\tT$, $\tH$, $\tP$ and
$\tF$. Let $D = (D_{ij})$ be the displacement matrix, where $D_{ij}$
is the set of relative positions of tiles of type $j$ in supertiles of
type $i$, as determined via the control points. With $\vL_i$ denoting
the control points of type $i$ in a fixed point tiling (for instance
the one generated by a hexagon with its control point at the origin),
one then has
\[
  \vL^{}_i \, = \: \bigcup_{j}^{.} \phi^2 \vL^{}_j
  \ts + D^{}_{ij} \ts ,
\]
where $A+B$ denotes the Minkowski sum of two sets, and the unions on
the right-hand side are disjoint.  

Under the $\star$-map, and upon taking closures, one obtains the
window iterated function system (IFS)
\[
    W^{}_i \, = \: \bigcup_j \phi^{-2} W^{}_j \ts + D_{ij}^{\star}
\]
for $W_i = \overline{\vL_i^{\star}}$. Note that, due to having taken
the closure, the unions on the right-hand side need no longer be
disjoint. Since $\phi^2$ is a PV number, this IFS is
\emph{contractive} and defines a unique attractor
$(W_1,\ldots,W_{24})\subset (\mathcal{K}\RR^2)^{24}$, where
$\mathcal{K}\RR^2$ is the space of non-empty compact subsets of
$\RR^2$, equipped with the Hausdorff metric. Each
\[
  \oplam (W_i) \, \defeq \, \{ x \in \ZZ[\xi,\phi]
  : x^{\star}\in W_i\}
\]
is a regular model set with $\vL_i \subseteq \oplam (W_i)$ by
construction, where each window $W_i$ is a topologically regular
compact set with almost no boundary; see \cite[Sec.~3]{BG} and
references therein for details.  Some of the boundaries are straight
lines, while some others are fractals of Koch curve type, then with
Hausdorff dimension
\begin{equation}\label{eq:dim}
    d^{}_{\mathrm{H}} \, = \, 
    \frac{\log \bigl( 2 + \sqrt{3}\, \bigr)}{2 \ts \log (\phi)}
    \, \approx \, 1.368 {\ts\ts} 376 {\ts\ts} 49  \ts ,
\end{equation}
as determined by standard methods, see Figure~\ref{fig:Koch}, from 
the induced edge IFS. 

\begin{figure}
\centerline{\includegraphics[width=0.9\textwidth]{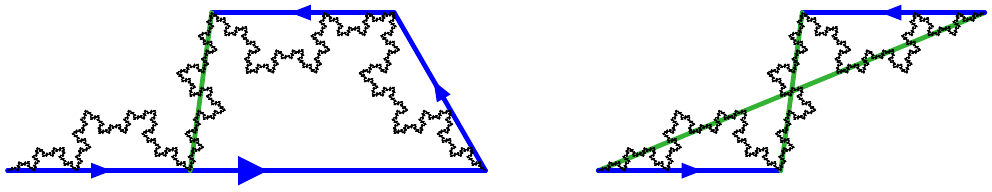}}
\caption{\label{fig:Koch}Koch curve type construction of the fractal
part of the window boundaries, for the $\tH$-tile (left) and the
$\tP$-tile (right). Here, every occurring corner is a projected 
lattice point, so the fractal curves contain lots of them.}
\end{figure}

We now provide the alternate proof that the CAP tiling has pure-point
dynamical spectrum. We will show that the $\vL_i$ have the same
density as the model sets $\oplam (W_i)$, thereby proving pure-point
diffraction spectrum and hence pure-point dynamical spectrum via the
equivalence theorem from \cite{LMS,BL}. Since the $\vL_i$ are
disjoint, it suffices to show that
\[
  \dens (\vL_1\cup\ldots\cup\vL_{24})
    \, = \, \dens \bigl( \oplam (W) \bigr)
\]
for $W=W_1\cup\ldots\cup W_{24}$, the window of
Figure~\ref{fig:window}, where we know that the $W_i$ are mutually
interior-disjoint.

The value of the density of the set of tiles or control points of the
inflation tiling can also be computed as follows.  The relative
frequencies of the four types of tiles are obtained from the right
Perron--Frobenius (PF) eigenvector of the meta-tile inflation matrix.
Normalized such that these frequencies add up to $1$, we get for the
tiles $(\tT,\tH,\tP,\tF)$ the frequency vector
\[
     \bs{f} \, =  \, \bigl(\tfrac{5}{3}-\phi,
    \tfrac{1}{3} , 2\,\phi - 3 , 2-\phi \bigr),
\]
which is independent of the choice of inflation. We remark that one
third of all tiles are hexagons.  For a stone inflation, the relative
areas of the four tile types could be extracted from the entries of
the left PF eigenvector. Here, however, these would be the relative
areas of the corresponding fractiles, whose absolute areas are not
immediately known. We therefore take the areas of the polygonal tiles
instead. With $a=\sqrt{3}/4$, the area of an equilateral triangle of
edge length $1$, we obtain, after some algebra, the vector of tile
areas
\[
     \bs{v} \, = \, a \bigl( 9+9\,\phi, 15+27\,\phi, 14+22\,
     \phi, 15+23\,\phi \bigr) .
\]
For this computation, we have used the lengths of the tile edges 
according to
\[
    (A,B,X,\fe,L) \, = \, (3\ts\phi,3\ts\phi,1,\phi+\xi, 2\ts\phi) \ts .
\]
The average area per tile then becomes $\bs{v}\cdot\bs{f}$, and its
inverse is the tile density.  However, we first want to rescale the
tiling such that the generator $3\phi+2-\xi$ of the return module $R$
has unit length, so that we have to multiply the density by
$|3\phi+2-\xi|^2=6(2+3\phi)$. At this new scale, we obtain a tile
density
\[
  \rho^{}_2 \, = \, \frac{|3\phi+2-\xi|^2}{\bs{v}\cdot\nts\bs{f}}
         \, = \, \myfrac25\sqrt{3}\,\phi^2.
\]

Since $\rho^{}_1 = \rho^{}_2$, the set of control points of the
inflation tiling is indeed the model set $\oplam (W)$, possibly up to
a difference set of zero density. The autocorrelation measures of both
sets agree, because the addition or subtraction of sets of zero
density does not matter \cite[Rem.~9.14]{TAO}, and the same applies to
the diffraction of Dirac combs on $\oplam (W)$ that are weighted
according to the point classes, in comparison to the correspondingly
weighted Dirac combs on the control point sets.  The diffraction
measures of both situations are thus the same, and pure point.  Since
the hulls of both sets give rise to ergodic dynamical systems, the
equivalence theorem \cite{LMS,BL} implies that the tiling dynamical
system has pure-point dynamical spectrum, and both systems have
eigenfunctions with continuous representatives. In other words,
topological and measure-theoretic point spectra are the same.

\begin{remark}
  The equality of $\rho^{}_{1}$ and $\rho^{}_{2}$ is also a
  consistency check of the correct choice of lattice for the CPS.  The
  result now gives the group $\RR^4/\cL\simeq \mathbb{T}^4$ as the
  \emph{maximal equicontinuous factor} (MEF) of our tiling space.  On
  the level of the control points, due to the nature of the windows,
  the sets $\vL_i$ and $\oplam (W_i)$ can only differ in points whose
  $\star$-image lies on a window boundary. If this happens (as it does
  here), the model set description gives the \emph{union} of all
  elements in the fiber over the corresponding point of the MEF. They
  are the \emph{singular} points, which have a slightly more
  complicated nature than usual, due to the partly fractal nature of
  the window boundaries.  \exend
\end{remark}

The dynamical spectrum now is the projection of $\cL^*$, the dual
lattice of the Minkowski embedding $\cL$ of $\ZZ[\xi,\phi]$, into
direct space, which is the dual module. Here, one finds
\begin{equation}\label{eq:spec}
    \ZZ[\xi,\phi]^* \, = \: \bigl( \ZZ[\phi] \cdot \ZZ[\xi]\bigr)^{*} 
    \, = \: \ZZ[\phi]^* \nts \nts \cdot \ZZ[\xi]^* \, = \: 
    \frac{\ZZ[\phi]}{\sqrt{5}} \cdot \frac{2 \ts \ii}{\sqrt{3}\ts}
    \ts \ZZ[\xi] \, = \, \frac{2 \ts \ii}{\sqrt{15}\ts }
    \ts \ZZ[\xi,\phi] 
\end{equation}
by a standard calculation, which gives a scaled and rotated copy of
$\ZZ[\xi,\phi]$ as claimed.

This concludes the proof of Theorem~\ref{main2} and also that
of Lemma~\ref{lem:pp}.

\begin{remark}
  The total window for \emph{all} control points is a regular hexagon,
  whereas the window for the $\tF$-tile control points consists of $6$
  regular triangles. By the general criteria for MLD relations of
  regular model sets in the same CPS, see \cite[Rem.~7.6]{TAO}, we
  thus see that the $\tF$-type points can locally be derived from
  knowing all points (without color), and vice versa.  All other
  subwindows partly have fractal boundaries. This means that control
  points of other tile types cannot be derived by local means from the
  (uncolored) positions of all control points.  \exend
\end{remark}

Let us say a bit more on the diffraction measure, formulated for the
control point set of the CAP tiling. Each of the $24$ Delone sets
$\vL_i$ gives rise to a Dirac comb
$\delta^{}_{\!\vL_i} \defeq \sum_{x\in\vL_i} \delta_x$, and thus to
the corresponding Fourier--Bohr coefficient
\[
  a^{}_{i} (k) \, \defeq  \lim_{s\to\infty}
  \myfrac{1}{\vol (B_s (0))}
  \sum_{x\in\vL_i \cap B_s (0)} \ee^{-2\pi\ii k x},
\]
which exists, for any $k\in\RR^2$, due to ergodicity.  Here,
$B_{s} (0)$ denotes the closed disk of radius $s$ around $0$. The
Fourier--Bohr coefficient is an area-averaged exponential sum, with
$a_{i} (0) = \dens (\vL_i) = \dens (\cL)\ts \vol (W_i)$ from our above
analysis. By the general spectral result for regular model sets,
compare \cite[Thm.~9.4]{TAO}, we get
\begin{equation}\label{eq:FB}
  a^{}_{i} (k) \, = \, \frac{\dens (\vL_i)}{\vol (W_i)}
  \, \widehat{1^{}_{W_i}} (-k^{\star}) 
\end{equation}
for any $k$ in the dynamical spectrum, and $a^{}_{i} (k) =0$
otherwise. Here, $1_A$ denotes the characteristic function of the set
$A$, and $\widehat{1_A}$ its Fourier transform.

If we now consider the weighted Dirac comb
$\omega = \sum_{i} h^{}_i \ts \delta^{}_{\!\vL_i}$, which is
pure-point diffractive, one obtains the diffraction measure
\[
  \widehat{\gamma} \: = \sum_{k \in L^{\circledast}}
  \big\lvert h^{}_{1} a^{}_{1} (k) + \ldots +
  h^{}_{24} a^{}_{24} (k) \big\rvert^2 \delta^{}_{k} \ts ,
\]
which is a pure-point measure supported on the Fourier module
$L^{\circledast} = \pi (\cL^{*})$, with $\pi$ as in
Eq.~\eqref{eq:CPS}.  The module $L^{\circledast}$ equals the dynamical
spectrum from above. The validity of this formula follows from the
fact that a primitive inflation tiling leads to Delone sets that have
the phase consistency property \cite{TAO,BGM,BG}. With little effort,
this can be generalized to other choices of control points, and to
convolutions with various profile functions.

\begin{remark}
  While the formula in \eqref{eq:FB} is a major result in the theory
  of regular model sets, its calculation can become complicated, in
  particular for windows with fractal boundaries. Here, the windows
  $W_i$ are still such that a numerical approach (as in \cite{Soc})
  gives reasonable approximations. An exact calculation is also
  possible, via the compactly converging Fourier cocycle developed in
  \cite[Sec.~4]{BG}, which works here in exactly the same way. A
  simple example, also in comparison to a numerical approximation, is
  worked out in \cite{BG-frac}.  \exend
\end{remark}

This is the standard approach for regular model sets, which directly
applies to the CAP tiling and its control points. To transfer any of
this to the other members of the Hat family of tilings, we need to
connect them to the same cut-and-project scheme, as we do next.

\subsection{Reprojections}

Imagine a Delone set produced like the control points of the CAP
tiling, with exactly the same total space and exactly the same window,
only with a different projection from the total space to physical
space. For this, we need some notation.  If $z\in \cL$ is an element
of our embedding lattice, we write it as $z = (x, x^{\star})$, where
$x=\pi(z)$ and $x^{\star}=\pi^{}_{\text{int}}(z)$ are the orthogonal
projections of $z$ into direct (or physical) and internal space,
respectively.

If $L$ is an arbitrary linear map from $\RR^2$ to $\RR^2$, then 
\begin{equation}
  \pi'(z) \, = \, \bigl( \pi + L\circ \pi^{}_{\text{int}} \bigr)(z)
  \, = \, x + L(x^{\star}) 
\end{equation}
is another projection from $\RR^4$ to $\RR^2$. In fact, \emph{every}
projection from the total space $\RR^4$ to direct space $\RR^2$ can be
written in this way. Changing projections is exactly the same as
adding a linear map from internal space to direct space, which is a
standard situation in the treatment of deformed model sets; see
\cite[Ex.~9.9]{TAO} and references given there.

The coordinates of $x^{\star}$ are weakly pattern-equivariant
functions on our Delone set.  If two points have identical
neighborhoods, their star coordinates are close but not identical.
However, the difference $x_1^{\star}-x_2^{\star}$, where $x^{}_1$ and
$x^{}_2$ are nearby points in our Delone set, is \emph{strongly}
pattern-equivariant. This means that the coboundaries of the
coordinates of $x^{\star}$ (viewed as $1$-cochains on the tiling
obtained by connecting the vertices of our Delone set) are linearly
independent, asymptotically negligible classes in
$\check{H}^1(\Omega_{\cT}, \RR )$. The span of the linear functions
from internal space to direct space is then a $4$-dimensional space of
asymptotically negligible classes in
$\check{H}^1(\Omega_{\cT}, \RR^2)$. However,
$\check{H}^1(\Omega_{\cT}, \RR^2)$ is only $4$-dimensional, so
\emph{all} asymptotically negligible classes are obtained in this way.

By reprojecting the CAP Delone set in arbitrary ways, we obtain Delone
sets whose shape classes differ from the CAP shape class by arbitrary
asymptotically negligible elements of
$\check{H}^1(\Omega_{\cT}, \RR^2)$. This means that every tiling that
is topologically conjugate to the CAP Delone set is MLD to such a
reprojection of the CAP Delone set, as claimed in Theorem
\ref{main2a}.  \smallskip

\begin{figure}
\begin{center}
\includegraphics[width=0.9\textwidth]{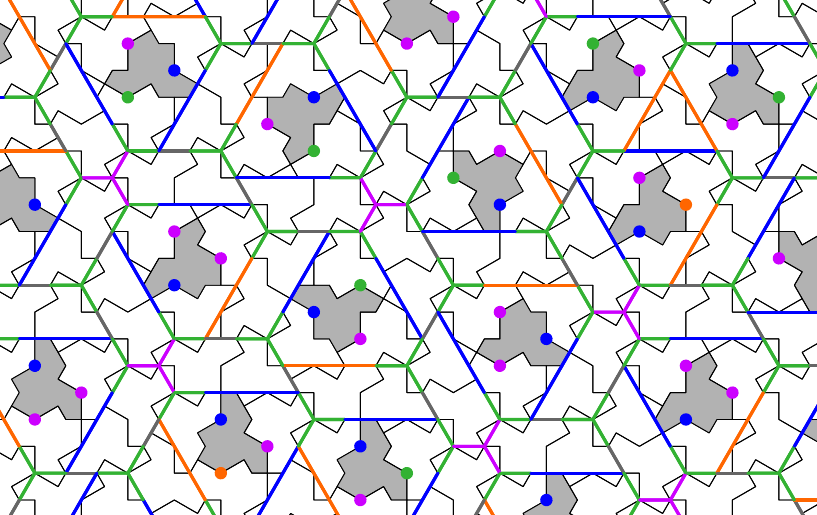}
\end{center}
\caption{\label{fig:Hats-meta-dots} Hat tiling with superimposed
  meta-tiles and control points.  The latter are obtained by
  reprojection of the control points of the CAP tiling.  }
\end{figure}

\begin{remark}\label{rem:not}
  Under topological conjugacy, singular points of the MEF must be
  mapped onto each other, preserving the cardinality of the fibers.
  The set of points in the MEF having a singular fiber (with more than
  one element) is given by
  $S = (\RR^2 \times \partial W) / \mathcal{L}$, where $\partial W$ is
  the set of all boundary points of the window (including internal
  boundaries), and $\RR^2$ are the translations in the direction of
  direct space. This is a subset of measure zero of the MEF
  $(\RR^2\times\RR^2) / \mathcal{L}$. Due to the fractal boundaries, a
  reflection preserving the total window cannot also preserve the set
  $S$. This means that translations of the hull in direct space
  direction cannot commute with a reflection, for any direction of
  projection, even if they do so for the MEF itself. Hence, a
  reflection can preserve the spectrum under certain conditions, but
  it cannot be a topological conjugacy.  \exend
\end{remark}

As an application, we determine the reprojected control points for the
Hat tiling. Figure~\ref{fig:Hats-meta-dots} shows a Hat tiling with
superimposed meta-tiles, along with their control points. These
control points are determined as follows. Let $\{t_1, t_2, t_3, t_4\}$
be a basis of the return module $R$ of the CAP tiling, consisting of
actual return vectors. For each $t_i$, one then finds a patch in a CAP
tiling containing two translation-equivalent tiles which are a vector
$t_i$ apart. There is then an equivalent patch in the Hat meta-tile
tiling, where the two equivalent tiles are a vector $t_i^{\ts\prime}$
apart. For instance, on the right of Figure~\ref{fig:Hats-meta-dots},
there are three hexagon meta-tiles with an orange edge at their lower
right. Their distances, together with the corresponding preimages in
the CAP tiling, could be used as pairs $(t^{}_i, t_i^{\ts\prime})$
(complemented with rotated such pairs).  The map
$t^{}_i \mapsto t_i^{\ts\prime}$ can now be extended linearly to a map
from the return module of the CAP tiling (of rank $4$) to the return
module of the Hat tiling (of rank $2$). This map is the reprojection
map for the control points.  Each control point of a CAP tiling is
mapped to a unique control point of the Hat tiling. From the (colored)
control point set of the Hat tiling, the Hat tiling itself can be
locally reconstructed.  Consequently, the Hat tiling and the colored
Delone set with the reprojected control points are MLD. Note that if
the control points are distributed over several orbits under the
return module, the reprojection procedure becomes considerably more
involved.

\section{Comparing different shapes}\label{sec:compare}

In this section, we compare several versions of the Hat tiling, in
particular the CAP tiling and the four named tilings from \cite{Hat}
whose meta-tile vertices live on a lattice: the Chevron, the Comet,
the Turtle and the original Hat. We will compute the expansions and
rotations needed to make these versions topologically conjugate and
see how their underlying lattices are rotated relative to one
another. We will compute the return module for each one and see how
its dual sits inside the already computed spectrum of the CAP tiling.
Note that we work here with the coordinates of the \emph{original} Hat
and meta-tiles, before rescaling the return module as for the
computation of the window in Figure~\ref{fig:window}.

Finally, we consider arbitrary elements of the Hat family considered
in \cite{Hat}, with all tile edges either being vertical or
horizontal, up to rotation by multiples of $60$ degrees. For each such
shape, the spectrum for the resulting tiling by meta-tiles is rotated
relative to the coordinate axes. The spectrum for the other way to
tile the plane, using a hierarchical system of anti-meta-tiles, is
rotated in the opposite direction by the same angle. As long as the
angle of rotation is not a multiple of $30$ degrees, the two rotated
patterns are different. That is, it is usually possible to distinguish
between the two LI classes of possible tilings by spectral means. This
is the case for all but two members of the Hat family: the Turtle and
a shape that has not yet been named.

The geometry of the Hat tile is shown in \cite[Fig.~3.2]{Hat}. In the
basic shape, the orange edges have length $\sqrt{3}/2$, while the
black edges have length $1/2$. More generally, we can let the orange
and black edges involve displacements by complex numbers $\alpha$ and
$\beta$. The Hat is then a polygon whose displacements are
\[
   \alpha, \, \xi \alpha, \, -\ii\xi \beta, \, -\ii\xi^2\beta, \, 
   -\xi^2 \alpha, \, -\xi \alpha, \, -\ii\xi \beta, \, -\ii\beta, \, 
   -\alpha, \, \xi^2\alpha, \, \ii\xi^2\beta, \, 2 \ii \xi \beta, \, 
   \text{and } \ii\beta \ts . 
\]
Varying $\alpha$ and $\beta$ gives our (real) $4$-parameter
family of shape changes that respect rotational symmetry.

In order for an anti-Hat to fit next to a Hat, the edges of the
anti-Hat must also involve $\alpha$ and $\beta$. However, the
reflection of the Hat about the horizontal axis has edge vectors
$\bar \alpha$ and $-\ii \bar \beta$ times powers of $\xi$. That is,
the anti-Hat can only be the reflection of the Hat about a horizontal
axis if $\alpha$ and $\beta$ are real. This is why the authors of
\cite{Hat} only considered real ratios $\alpha : \beta$. We will
eventually consider arbitrary ratios $\alpha:\beta$, but for now we
consider four named versions in which the vertices of the meta-tiles
lie on the triangular lattice $\ZZ[\xi]$, and the self-similar CAP
tiling as well as Socolar's version from \cite{Soc}.

\begin{itemize}
\item When $\alpha=1$ and $\beta=0$, the Hat degenerates to a Chevron,
  which is the union of four equilateral triangles with a common
  vertex. After rotating by $60$ degrees (which is merely changing
  which orientation is considered standard), the shape
  $(A, B, X, \fe, L)$ is given by the row vector
\[
  (1+\xi, 1+\xi, 1, 1, 0) \, = \, \bmu^{}_1+\bmu^{}_2 \ts .
\]

\item The original Hat has $\alpha=\sqrt{3}/2$ and
  $\beta=1/2$. After rotating by $90$ degrees, our shape vector is
\[
  (3, 3, 1, 1, 1) \, = \, \bmu^{}_1 + \bmu^{}_2 + \bmu^{}_3
  \, \sim \, \myfrac{3}{2} \bmu^{}_1 + \bmu^{}_2 \ts ,  
\]
   where $x\sim y$ means that $x$ and $y$ are cohomologous.

\item The Turtle has $\alpha=1/2$ and $\beta=\sqrt{3}/2$. After
  rotating by $120$ degrees, our shape vector is
\[
  \bigl( 2(1+\xi^5), 2(1+\xi^5), 1, 1, 1+\xi^5 \bigr)
     \, = \, \bmu^{}_1+\bmu^{}_2 + (1+\xi^5) \bmu^{}_3
     \, \sim \, \myfrac{3+\xi^5}{2} \bmu^{}_1 + \bmu^{}_2 \ts .
\]

\item When $\alpha=0$ and $\beta=1$, the Hat degenerates to a
  Comet, which is the union of eight equilateral triangles.  After
  rotating by $90$ degrees, $(A,B,X,\fe,L)$ equals
\[
  (\xi, \xi, \xi, \xi, 2) \, = \, \xi(\bmu^{}_1+\bmu^{}_2) + 2 \bmu^{}_3
  \, \sim \, (\xi+1)\bmu^{}_1 + \xi \bmu^{}_2 \ts .
\]

\item As noted earlier, the CAP tiling has shape vector
\[
  (3\phi, 3\phi, 1, \phi+\xi, 2\phi) \, = \,
  \bmu^{}_1 + (\phi+\xi)\bmu^{}_2  + 2\phi \bmu^{}_3 \, \sim \,
  (1+\phi) \bmu^{}_1 + (\phi+\xi)\bmu^{}_2 \ts.
\]
\item We note in passing that the shape change used to describe
  Socolar's golden Key tiles \cite{Soc} is described by the shape
  vector
\[
\begin{split}    
  (3\phi+2-\xi, 3\phi+2-\xi,0, \phi+\xi, 2\phi+2) \, & = \,  
  (\phi+\xi)\bmu^{}_2  + (2\phi+2) \bmu^{}_3 \\
  & \sim \, (1+\phi) \bmu^{}_1 + (\phi+\xi)\bmu^{}_2 \ts ,
\end{split}
\]
which is thus cohomologous to that of the CAP tiling.  This implies
that the golden Key tiling is MLD to the CAP tiling.
\end{itemize}

Let us next compute the return modules for each tiling.  The return
module is generated by collections of edges that correspond to loops
in the AP complex. That is, we want chains that are in the kernel of
the boundary map $\partial^{}_1$. A basis for this kernel is
\begin{enumerate}
\item $\; A + r^3 B$ (times powers of $r$). 
\item $\; A + (r^3-r^5)X$ (times powers of $r$). 
\item $\; (1-r^2) \fe$ (times powers of $r$), and 
\item $\; (r+1) \bigl( (1-r^3)X + L \bigr)$ (times powers of $r$). 
\end{enumerate}
However, we must have
\[
  A \, = \, B \, = \, (1+\xi^5)X + (\xi + \xi^2) \fe + (1+ \xi^5) L \ts .
\]
With this,  our return module is generated by 
\[
  (\xi + \xi^2) \fe + (1+\xi^5)L, \quad (1-\xi^2) \fe,
  \; \hbox{ and } \, (\xi+1)(2X+L) \ts ,
\]
times powers of $\xi$. By taking linear combinations, this reduces to 
\[
  (\xi+1) \fe, \quad (\xi+1)L, \; \hbox{ and } \, 2(\xi+1)X  .
\]
\begin{itemize} 
\item In the Chevron, Hat, Turtle and Comet, $X=\fe$ is a power of
  $\xi$, while $L$ is in $\ZZ[\xi]$. This makes the return module
  $(1+\xi) \ZZ[\xi]$.  These are the only members of the Hat family of
  tilings where the return module has rank 2.
\item In the CAP tiling, the generators are $(1+\xi)$ times 2,
  $\phi+\xi$ and $2\phi$.  We could just as well take the generators
  to be $(1+\xi)$ times 2, $\phi-\xi$ and $2\phi$. The resulting
  module is invariant under multiplication by $\xi$ and also by $\phi$,
  the latter because
\[
  \phi \ts (\phi-\xi) \, = \, 2 + \xi^5 (\phi-\xi) \ts ,
\]
which is to say that it is a $\ZZ[\xi, \phi]$-module. However, over
$\ZZ[\xi,\phi]$, the numbers $2$ and $2\phi$ are multiples of
$\phi-\xi$, since
\[
  (\phi-\xi)(\phi-\xi^5) \, = \, 2 \ts .
\]
Thus, our module is $(1+\xi)(\phi-\xi) \ZZ[\xi, \phi]$. Since
$\phi^2(1+\xi)(\phi-\xi)=3\phi+2-\xi=t^{}_{0}$, 
see Eq.~\eqref{eq:t-vec},
and $\phi^2$ is a unit
of $\ZZ[\xi, \phi]$, this agrees with our findings in Section~3.
 \end{itemize}

 Next, we determine the rescalings needed to make these tilings
 conjugate.  We wish to decompose the cohomology class $S$ of each
 shape vector in terms of the left eigenvectors
 $\ell^{}_\lambda = (\lambda -\xi)\bmu^{}_1 + 2 \bmu^{}_2$ of the
 substitution operator as in equation (\ref{eq:eigen}), where
 $\lambda = \phi^{\pm 2}$. We get the coefficient of
 $\ell^{}_{\lambda}$ by taking an inner product of $S$ with the
 right eigenvector $r^{}_{\lambda}$ and normalizing, thus obtaining
 the coefficient of $\ell^{}_{\lambda}$ as
\[
  \frac{S \cdot r^{}_{\lambda}}{\ell^{}_{\lambda} \cdot
    r^{}_{\lambda}} \, = \,
  \frac{S \cdot r^{}_{\lambda}}{(3-2\xi)\lambda + (3 \xi-2)} \ts .
\]
The values of the inner product $S \cdot r^{}_{\lambda}$ for
$\lambda=\phi^2$ and $\lambda=\phi^{-2}$ are shown in
Table~\ref{tab:compare}.

\begin{table}[t]
  \caption{A comparison of the tilings discussed in this paper. The
    second and third columns are the coefficients of $\bmu^{}_1$ and
    $\bmu^{}_2$, while the last two columns
    show the inner product of the shape class $S$ and the 
    right eigenvectors $r^{}_\lambda$
    from \eqref{eq:eigen}.
    \label{tab:compare}}
\renewcommand{\arraystretch}{1.5}
\begin{tabular}{|c|c|c|c|c|}\hline
  Tiling & $\bmu^{}_1$ & $\bmu^{}_2$ & $S \cdot r_{\phi^2}$
         & $S \cdot r_{\phi^{-2}}$ \\ \hline \hline
 Chevron & 1 & 1 & $\phi^2$
        & $\phi^{-2}$ \\ \hline
        Hat & 3/2 & 1 & $-\frac{1}{2} \xi \phi^2 (\phi-\xi^5)^3$
      & $\frac{1}{2} \xi \phi^{-2} (\phi-\xi)^3$ \\ \hline 
        Turtle & $(3+\xi^5)/2$ & 1 &$\frac{\sqrt{5}}{2} \ts
                            \xi^5\phi^2(\phi-\xi^5)$   
        & $\frac{\sqrt{5}}{2} \ts \xi^5 \phi^{-2} (\phi-\xi)$ \\ \hline
 Comet & $1 + \xi^5$ & 1 & $\phi^2 (\phi-\xi) $ &
        $-\phi^{-2} (\phi - \xi^5)$  \\ \hline  
  CAP & $\phi+1$ & $\phi+\xi$ & $\sqrt{5} \, \phi^2 $
       & 0 \\ \hline
\end{tabular}
\end{table}

From here on, we take the CAP tiling as our reference. To make the
other tilings topologically conjugate to it, we must multiply the
shape vectors of the Chevron, Hat, Turtle, and Comet by $\sqrt{5}$,
$\xi^2\sqrt{5}\ts (\phi-\xi)^3/4$, $\xi (\phi-\xi)$, and
$\sqrt{5} \ts (\phi - \xi^5)/2$, respectively. Note that the argument
of $\phi - \xi^5$ is $\arctan \bigl( \sqrt{3/5\ts}\, \bigr)$, which is
not a rational angle. With these rotations (and rescalings), the
underlying lattices of the Comet, Hat, Turtle, and Chevron have become
incommensurate. What connects them is the higher-rank return module of
the CAP tiling.  \smallskip

Let $L^{}_0=\ZZ [\xi]$ be the standard triangular lattice, so
$L_0^* = \frac{2\ts \ii}{\sqrt{3}} L_0$ is its dual. The spectrum of
the CAP tiling is the dual of the return module
\[ (1+\xi)(\phi-\xi) (L_0 \oplus \phi L_0) \, = \,
  (1+\xi)(\phi-\xi) \ZZ[\xi,\phi] \ts ,
\]
and is therefore 
\[
  \frac{(1+\xi^5)^{-1}(\phi-\xi^5)^{-1}}{\sqrt{5}}
  (L_0^* \oplus \phi L_0^*) \,  = \, 
  \frac{(1+\xi)(\phi-\xi)}{6 \sqrt{5}}
  \bigl( L_0^* \oplus \phi L_0^* \bigr) .
\]
For the computation of the window (Figure~\ref{fig:window}), we have
divided by the prefactor $(1+\xi)(\phi-\xi)$ (and the unit $\phi^2$).

Of course, the duals to the return modules of the Chevron, Hat, Turtle
and Comet tilings must be in the spectrum of those tiling spaces, and
so must be in the spectrum of the CAP tiling.  We check that this is
indeed the case, despite the incommensurate nature of those return
modules.
\begin{itemize} 
\item The return module of the rescaled Chevron is
  $\sqrt{5}(1+\xi) L_0$, whose dual is
  $\frac{1}{3\sqrt{5}} (1+\xi) L_0^*$.
\item The return module of the rescaled Hat is
  $\frac{\sqrt{5}}{4} (1+\xi)(\phi - \xi)^3 L_0$, whose dual is \\
  $\frac{1}{6\sqrt{5}} (1+\xi)(\phi - \xi)^3 L_0^*$. 

\item The return module of the rescaled Turtle is
  $(1+\xi)(\phi-\xi) L_0$, whose dual is \\
  $\frac{1}{6} (1+\xi)(\phi-\xi) L_0^*$.  

\item The return module of the rescaled Comet is
  $\frac{\sqrt{5}}{2} (1+\xi)(\phi - \xi^5) L_0$, whose dual is \\
  $\frac{1}{3\sqrt{5}} (1+\xi)(\phi - \xi^5) L_0^*$. 
\end{itemize}

To see that these are all contained in
$\frac{1}{6\sqrt{5}} (1+\xi)(\phi - \xi) (L_0^* \oplus \phi L_0^*)$,
we note that
\begin{equation}
  \begin{split}
  \frac{1+\xi}{3 \sqrt{5}}
  & \, = \, (\phi-\xi^5)\, \frac{(1+\xi)(\phi - \xi)}{6\sqrt{5}} \ts ,
   \\  \frac{(1+\xi)(\phi-\xi)^3}{6\sqrt{5}}
  & \, = \,   (\phi - \xi)^2 \, \frac{(1+\xi)(\phi - \xi)}{6 \sqrt{5}} \ts ,
   \\ \frac{(1+\xi)(\phi-\xi^5)}{6}
  & \, = \,  (2\phi-1) \, \frac{(1+\xi)(\phi - \xi)}{6\sqrt{5}} \ts ,
   \\ \frac{(1+\xi)(\phi-\xi^5)}{3\sqrt{5}}
  & \, = \,  (\phi-\xi^5)^2 \, \frac{(1+\xi)(\phi - \xi)}{6\sqrt{5}} \ts  .
  \end{split}
\end{equation}

Finally, we consider the spectra of the entire family of Hat-like
tilings considered in \cite{Hat} and prove Theorem~\ref{thm:reflect}.
For this purpose, we express everything in terms of the (real)
parameters $\alpha$ and $\beta$.  In order to consider all of these
models on the same footing, we declare the standard orientation of the
tile to be the one where the edges of length $\alpha$ and $\beta$ are
horizontal (resp.\ vertical), up to rotation by multiples of $60$
degrees. (For some of the named tilings, this disagrees with our
previous conventions.) With this choice, the displacements of the five
basic edges $A$, $B$, $X$, $\fe$ and $L$ in terms of $\alpha$ and
$\beta$ are
\begin{align}
  A \, = \,  B & \, = \,  (\xi+\xi^2)\alpha + 3\ts \ii
                 \beta \ts , \notag \\ 
  X \nts \, = \, \fe & \, = \,  \xi(\alpha + \ii \beta) \ts , \\ 
  L & \, = \,  2\ts \ii \beta \ts . \notag
\end{align}
Using Eq.~\eqref{eq:threemu}, our shape class is then 
\begin{equation}
  \begin{split}
  S \, & = \,   \alpha(\xi+\xi^2,\xi+\xi^2,\xi, \xi, 0 ) 
          + \ii \beta(3, 3, \xi, \xi, 2) \\[1mm]
    & = \,  (\alpha + \ii \beta) (\xi \bmu^{}_1 + \xi \bmu^{}_2) +
          2\ts \ii \beta \bmu^{}_3, 
  \end{split}
\end{equation}
which is cohomologous to
$ \xi(\alpha + \ii \beta) (\bmu^{}_1 + \bmu^{}_2) + \ii \beta
\bmu^{}_1 $.  Note that we can achieve any linear combination of
$\bmu^{}_1$ and $\bmu^{}_2$ by adjusting $\alpha$ and $\beta$
appropriately. This shows that all shape changes that preserve
rotational symmetry are MLD to changes in $\alpha$ and $\beta$.

We now compute how the spectrum of this tiling aligns with the
coordinate axes.  Taking the inner product of the shape class $S$ with
$r^{}_{\phi^2}$ from \eqref{eq:eigen} gives
\[
  \xi \phi^2(\alpha+i\beta) + \ii\beta(\phi^2-\xi) \, = \,
  \xi\phi^2 \bigl( \alpha + \ii\beta(\phi-\xi) \bigr) .
\]
That is, the $(\alpha,\beta)$ tiling is topologically conjugate to a
rescaling of the CAP tiling by the factor
$\xi \bigl( \alpha + \ii\beta(\phi-\xi) \bigr)/\sqrt{5}$. The factors
of $\xi$ and $\sqrt{5}$ are irrelevant, as the tilings are invariant
under rotation by $60$ degrees and we are keeping track of angles, not
scale.

The spectrum of the CAP tiling is itself rotated by the argument of
$(\phi-\xi)$ relative to the axes of the CAP tiling,\footnote{The
  factor of $(1+\xi)=\sqrt{3} \ii\ts \xi^5$ cancels the factor of
  $\ii$ in $L_0^* = \frac{2\ii}{\sqrt{3}} L_0$.}  so the spectrum of
the $(\alpha, \beta)$ tiling is rotated by the argument of
$\bigl( \alpha + \ii \beta(\phi - \xi) \bigr) (\phi-\xi)$ relative to
the original $x$ axis. A simple computation shows that
\begin{equation}
  \bigl( \alpha + \ii\beta(\phi-\xi) \bigr) (\phi - \xi)
  \, = \, \myfrac{1}{2} \bigl( \sqrt{5} (\alpha + \sqrt{3} \beta)
        + \ii (\beta - \sqrt{3}\alpha) \bigr).
\end{equation}
For the argument of this number to be a multiple of $30$ degrees, the
ratio
\[
  \frac{\beta-\sqrt{3} \alpha}{\sqrt{5} (\alpha + \sqrt{3} \beta)}
\]
must be $0$, $\pm 1/\sqrt{3}$, or $\pm \sqrt{3}$. 

Since both $\alpha$ and $\beta$ are non-negative, the ratio is always
between $-\sqrt{3/5}$ and $1/\sqrt{15}$, and so can never equal
$\sqrt{1/3}$ or $\pm \sqrt{3}$. That leaves $-1/\sqrt{3}$, which is
achieved when
\[
  \beta \, = \, \frac{\sqrt{5}-2}{\sqrt{3}} \alpha \ts ,
\]
and $0$, which is achieved when 
\[
  \beta \, = \, \sqrt{3}\ts \alpha \ts .
\] 
The shape corresponding to the first ratio is similar to the Chevron,
since $(\sqrt{5}-2)/\sqrt{3} \approx 0.136$ is small. This shape does
not (yet) have a name. The shape corresponding to the second ratio is
the Turtle. Although the two LI classes for the Turtle have the same
spectrum, the argument of Remark~\ref{rem:not} shows that they are
\emph{not} topologically conjugate, nor are the two LI classes of the
other shape.

This completes the proof of Theorem~\ref{thm:reflect}.

\section{Some comments and questions}\label{sec:final}

Unlike previous examples, which often were limit-periodic, the CAP
tiling is \emph{quasiperiodic} in the sense of mean almost-periodic
measures, which have recently been analyzed by Lenz, Spindeler and
Strungaru \cite{LSS}.  As such, its pure-point diffraction corresponds
to the mean almost periodicity of the (possibly color-weighted) Dirac
comb on the control points.  This quasiperiodic nature is preserved
for the reprojected point sets, and also the Hat tiling is thus
quasiperiodic in this sense.

Consequently, the diffraction measure of the Hat tiling is also pure
point. Since the control points live on a lattice, the corresponding
diffraction measure is the lattice-periodic repetition of a finite
motif (or block), with the dual lattice as lattice of periods, by an
application of \cite[Thm.~10.3]{TAO}. This is an interesting feature
in the light of the intrinsic aperiodicity, though well known from
other examples such as the aperiodic Fibonacci chain with two types of
weights on $\ZZ$, which then has a periodic diffraction with dense
peaks.

For the Hat, the motif can be represented as a finite point measure on
a fundamental domain of the dual lattice that consists of a dense set
of Dirac peaks. They then encode the modulation due to the
quasiperiodic nature of the tiling, as also analyzed numerically in
\cite{Soc}. Some more detailed analysis of this feature seems of
interest.

The quasiperiodic nature of the Hat tiling with its Euclidean internal
space leads to a deformation structure that differs fundamentally from
that of a limit-periodic system with its $p$-adic (or related)
internal space. Specifically, having a Euclidean internal space allows
for reprojections. Reprojections that preserve rotational symmetry
then allow for an infinite family of monotiles.

We close with several open questions concerning the development of 
monotiles. 
\begin{itemize} 
\item Do there exist other monotiles, not conjugate to the Hat, whose
  quasiperiodicity similarly implies the existence of an infinite
  family of monotiles?
\item Does there exist a disk-like monotile for which the tilings by
  isometric copies of the tile define a unique and aperiodic LI class?
\item Does there exist an aperiodic disk-like chiral
  monotile whose tilings only involve rotations and translations (but
  not reflections) of the basic tile?
\item Does there exist a disk-like chiral monotile for which the
  tilings by rotations and translations (but not reflections) of the
  tile define a unique and aperiodic LI class?
\end{itemize}

Note: Within a month of the initial release of this paper, Smith et al
\cite{Spectre} answered the third question to the affirmative. They
constructed a new aperiodic monotile, called the Spectre, whose
tilings use $12$ rotations of the basic tile but no
reflections. Tilings built from these tiles form two LI classes, each
$6$-fold rotationally symmetric, but rotated against each other by
$2\pi/12$. The Spectre tiling is quasiperiodic (in the sense we used
above) and admits shape conjugacies that respect $6$-fold rotational
symmetry \cite{BGMS}.  However, these deformations are not $12$-fold
rotationally symmetric and do \emph{not} yield an infinite family of
monotiles.

\section*{Acknowledgments}

It is our pleasure to thank Dirk Frettl\"{o}h, Craig Kaplan and Jan
Maz\'{a}\v{c} for discussions.  This work was supported by the German
Research Council (Deutsche Forschungsgemeinschaft, DFG) under SFB-TRR
358/1 (2023) -- 491392403.  

\end{document}